\documentclass[11pt]{article}

\usepackage{amsmath,amsfonts,amsbsy,amsgen,amscd,amsxtra,amstext,amsopn,amssymb}
\usepackage{latexsym}
\usepackage{theorem}
\usepackage{mathrsfs}
\usepackage{euscript}
\usepackage{graphicx}

{\theorembodyfont{\slshape}
\newtheorem{theorem}{Theorem}[section]
\newtheorem{proposition}[theorem]{Proposition}
\newtheorem{lemma}[theorem]{Lemma}
\newtheorem{corollary}[theorem]{Corollary}

}
{\theorembodyfont{\rmfamily}
\newtheorem{definition}[theorem]{Definition}

\newtheorem{remark}[theorem]{Remark}

}

\edef\qedrestoreat{\noexpand\catcode\lq\noexpand\@=\the\catcode\lq\@}
\catcode\lq\@=11

\ifx\protect\undefined\let\protect\relax\fi

\def\qed{\protect\@qed{$\qedsymbol$}}
\def\pushright{\protect\@pushright}

\def\QED{\protect\@qed{{\rm Q.E.D.}}}

\def\QEI{\protect\@qed{{\rm Q.E.I.}}}

\def\Proof{\protect\@Proof}\def\endProof{\protect\@endProof}%

\def\Proofof#1{\protect\@Proofof{#1}}\def\endProofof{\protect\@endProofof}%

\def\qedsymbol{\raisebox{-.2ex}{$\Box$}}

\def\TheWordProof{\em Proof.}
\def\TheWordProofof#1{\em Proof of #1.}

\def\ProofFont{}

\newif\ifAutoQED\AutoQEDfalse

\newif\ifNumberResults
\ifx\theorem@style\undefined
   \NumberResultstrue
   
\fi

\def\parag@pushright#1{{
    \parfillskip=0pt            
    \widowpenalty=10000         
    \displaywidowpenalty=10000  
    \finalhyphendemerits=0      
    \hbox@pushright             
    #1
    \par}}

\def\hbox@pushright{
    \unskip                     
    \nobreak                    
    \hfil                       
    \penalty50                  
    \hskip.2em                  
    \null                       
    \hfill                      
}%

\newif\if@qed\@qedfalse
\def\save@set@qed{\let\saved@ifqed\if@qed\global\@qedtrue}%
\def\restore@qed{\global\let\if@qed\saved@ifqed}

\def\@Proof{%
   \par\removelastskip\bigskip\penalty100
   \save@set@qed
   \noindent\ProofFont{\TheWordProof\enskip}%
}%
\def\@Proofof#1{%
   \par\removelastskip\bigskip\penalty100
   \save@set@qed
   \noindent\ProofFont{\TheWordProofof{#1}\enskip}%
}%
\def\@endProof{%
   \qed\restore@qed
   \penalty-100 \medskip
}
\def\@endProofof{%
   \qed\restore@qed
   \penalty-100 \medskip
}
\def\@qed#1{%
\if@qed                                 
     \global\@qedfalse
        \ifmmode\ifinner\pushright{#1}
        \else\eqno{\qedsymbol}\fi
        \else\pushright{#1}\fi%
\else\ifhmode\ifinner\else\par\fi\fi
\fi}
\def\@pushright#1{%
  {\ifvmode                             
       \null\hfill{#1}\par              
  \else\ifmmode\maths@pushright{\hbox{#1}}
       \else\ifinner\hbox@pushright{#1}
            \else\parag@pushright{#1}
  \fi  \fi  \fi
}}%
\def\maths@pushright#1{{%
  \ifinner
     \hbox@pushright{#1}%
  \else
     \eqno#1
     \def\]{$$\ignorespaces}
  \fi
}}%
\newcommand{\comment}[1]{}
\newcommand{\conv}{\mathsf{conv}}
\newcommand{\cone}{\mathsf{cone}}
\newcommand{\vol}{\mathsf{vol}}

\newcommand{\dist}{\mathsf{dist}}
\newcommand{\Fe}{\mathcal{F}}
\newcommand{\In}{\mathcal{I}}

\newcommand{\bE}{\mathbb{E}}
\newcommand{\inte}{\mathsf{int}}
\newcommand{\EE}{\mathcal{E}}
\newcommand{\bI}{\mathbf{1}}
\def\a{\alpha}
\def\b{\beta}
\def\d{\delta}
\def\s{\sigma}
\def\r{\rho}
\def\ph{\varphi}
\def\e{\varepsilon}
\def\R{\mathbb R}

\def\N{\mathbb N}

\def\<{\langle}
\def\>{\rangle}
\def\Oh{\mathcal O}
\def\CC{{\mathscr C}}

\def\Prob{\mathop{\mathsf{Prob}}}

\def\Proj{\mathbb{P}}

\def\scone{\mathsf{sconv}}
\def\dP{d_{\Proj}}
\def\Tusn{T^{\perp}}
\def\kc{{\cal C}}

\newcommand{\lan}{\left\langle}
\newcommand{\ran}{\right\rangle}

\newcommand{\veps}{\e}
\newcommand{\scC}{\CC}

\newcommand{\msD}{{\mathscr D}}
\newcommand{\scK}{{\mathscr K}}
\newcommand{\diam}{\mathsf{diam}}
\newcommand{\dnew}{\omega}
\newcommand{\cnew}{C}
\def\aa{\overline{a}}

\newenvironment{proofsketch}{\trivlist\item[]\emph{Proof (Sketch)}.}%
{\unskip\nobreak\hskip 1em plus 1fil\nobreak$\Box$
\parfillskip=0pt%
\endtrivlist}

\def\msD{\mathscr{D}}
\def\diam{\mathrm{diam}}
\def\sp{{s}}

\begin{document}

\title{\bf Robust Smoothed Analysis of a Condition Number for Linear Programming}

\author{Peter B\"urgisser$^\ast$ and Dennis Amelunxen\thanks{Institute of Mathematics,
University of Paderborn, Germany. Partially supported
by DFG grant BU 1371/2-1 and
DFG Research Training Group on Scientific Computation GRK 693
(PaSCo GK).}\\
University of Paderborn\\
\{pbuerg,damelunx\}@math.upb.de
}

\maketitle


\begin{abstract}
We perform a smoothed analysis of the GCC-condition number~$\scC(A)$
of the linear programming feasibility problem
$\exists x\in\R^{m+1}\ Ax < 0$.
Suppose that $\bar{A}$ is any matrix with rows $\aa_i$ of euclidean norm~$1$
and, independently for all~$i$,
let $a_i$ be a random perturbation of $\aa_i$ following the uniform distribution
in the spherical disk in $S^m$ of angular radius $\arcsin\s$ and centered at $\aa_i$.
We prove that $\bE(\ln\scC(A)) =\Oh (mn/\s)$.
A similar result was shown for Renegar's condition number and Gaussian perturbations
by Dunagan, Spielman, and Teng [arXiv:cs.DS/0302011].
Our result is robust in the sense that it easily extends to
radially symmetric probability distributions supported on a spherical disk of radius
$\arcsin\s$,
whose density may even have a singularity at the center of the perturbation.
Our proofs combine ideas from a recent paper of
B\"urgisser, Cucker, and Lotz (Math.\ Comp.\ 77, No. 263, 2008)
with techniques of Dunagan et al.
\end{abstract}

\smallskip


\smallskip


\section{Introduction}

A distinctive feature of the computations considered in numerical
analysis is that they are affected by errors. A main character in
the understanding of the effects of these errors is the {\em
condition number} of the input at hand. This is a positive number
measuring the sensitivity of the output with respect to small
perturbations of the input. The best known condition number is that
for matrix inversion and linear equation solving, which takes the
form $\kappa(A)=\|A\|\,\|A^{-1}\|$ for a square matrix~$A$.
Condition numbers not only occur in round-off analysis, but also
appear as a parameter in complexity bounds for a variety of
iterative algorithms in linear algebra, linear and convex
optimization, and polynomial equation solving. Yet, condition
numbers are not easily computable. As a way out for this situation,
Smale suggested to assume a probability measure on the set of data
and to study the condition number of this data as a random variable.
Examples of such results abound for a variety of condition numbers.
For more details and references we refer to Smale's
survey~\cite{Smale97} and the recent survey~\cite{buer:09a}.

Renegar~\cite{rene:94,rene:95b,rene:95a} was the first to realize
that the computational complexity of linear programming problems can
be bounded by a polynomial in the number of variables and inequalities
and a certain condition measure of the input.
This condition measure is given by the inverse distance of the input
to the set of ill-posed systems.
More specifically, it is well known that for a
given matrix $A\in\R^{n\times (m+1)}$, either the system $Ax<0$ or
its dual system $A^Ty=0,\ y>0$ have a solution, unless we are in an
ill-posed situation. The (homogeneous) linear programming
feasibility problem is to decide this alternative for given $A$ and
to compute a solution of the corresponding system.
A primal-dual interior point method is used
in~\cite{cupe:02} to solve the linear programming feasibility
problem within
\begin{equation}\label{eq:cba-cupe}
 \Oh\big(\sqrt{m+n}\,(\ln(m+n) + \ln\CC(A))\big)
\end{equation}
iterations, with each step costing at most $\Oh((m+n)^3)$ arithmetic
operations. Here, the {\em GCC-condition number}~$\CC(A)$
is a variant of Renegar's condition
number introduced by Goffin~\cite{goff:80}, and later generalized by
Cucker and Cheung~\cite{ChC:01} (see \S\ref{se:GCC} for the definition).
The advantage of $\CC(A)$ over Renegar's condition number
is that this quantity can be neatly characterized in terms
of spherical geometry, which greatly facilitates its probabilistic
analysis.

Thus the running time of the primal-dual interior
point method used in~\cite{cupe:02} is controlled by~$\CC(A)$.
Understanding the average-case behaviour of this algorithm
therefore boils down to studying $\ln\CC(A)$ for
random inputs~$A$.
Motivated by this observation, a lot of efforts have been devoted to the
{\em average-case analysis} of the random variable $\CC(A)$,
i.e., to compute the expected value (or the distribution
function) of $\ln\CC(A)$ for random matrices~$A$.
In most cases, the matrices $A$ are assumed to have random entries which are
i.i.d.\ standard normal.
We remark that since the condition number $\CC(A)$ is
multi-homogeneous in the rows~$a_i$ of $A$, this is equivalent to
considering $\CC(A)$ in the case where $a_1,\ldots,a_n$
are i.i.d.\ uniformly distributed in unit sphere
$S^m := \{x\in\R^{m+1}\mid \|x\|=1\}$.

The papers dealing with the average analysis of $\CC(A)$ are easily
summarized. A bound for $\bE(\ln\CC(A))$ of the form $\Oh(\min\{n,
m\ln n\})$ was shown in~\cite{ChC01}. This bound was improved
in~\cite{CW01} to $\max\{\ln m, \ln\ln n \}+\Oh(1)$ assuming that
$n$ is moderately larger than $m$. Still, in~\cite{ChCH:05}, the
asymptotic behavior of both $\CC(A)$ and $\ln\CC(A)$ was
exhaustively studied and these results were extended in~\cite{HM:06}
to matrices $A\in (S^m)^n$ drawn from distributions more general
than the uniform. Finally, in~\cite{BCL:08a}, the exact distribution
of $\CC(A)$ conditioned to $A$ being feasible was found and
asymptotically sharp tail bounds for the infeasible case were given.
In particular, it was shown that $\bE(\ln\CC(A))=\Oh(\ln m)$.
Our method yields another proof of this result~(Cor.~\ref{cor:average}).

\subsection{Smoothed analysis}

The problem of average-case analysis is that its results
strongly depend on the distribution of the inputs,
which is unknown, and usually assumed to be Gaussian
for rendering the mathematical analysis feasible.
Spielman and Teng~\cite{ST:02,ST:03,ST:04} suggested in 2001 the concept
of {\em smoothed analysis of algorithms}, which is a new form of analysis of algorithms
that arguably blends the best of both worst-case and average-case.
They used this new framework to give a more compelling explanation of the speed
of the simplex method (for the shadow-vertex pivot rule).

The general idea of smoothed analysis is easy to explain.
Let $T\colon \R^p\rightarrow \R_{+}\cup\{\infty\}$ be any function
(measuring running time, log of condition numbers etc.).
Instead of showing
``it is unlikely that $T(a)$ will be large,''
one shows that
``for all $\aa\in\R^p$ and all slight random perturbations $a$ of $\aa$,
 it is unlikely that $T(a)$ will be large.''
If we assume that $a$ is multivariate normal with mean $\aa$ and
variance~$\s^2$, in short $a\in N(\aa,\s^2)$,
then the goal of a smoothed analysis of~$T$
is to give good estimates of
$$
 \sup_{\aa\in\R^p}\ \Prob_{a \in N(\aa,\s^2)}\{ T(a) > \e^{-1}\} .
$$
In a first approach one may focus on bounds on the expectations.

For many situations of interest, it turns out that in smoothed analysis,
there is only a weak dependence on the chosen model of random perturbations.
A first formal evidence of this robustness phenomenon was given by
Cucker, Hauser, and Lotz~\cite{CHL:09} for certain conic condition numbers,
stated below as Theorem~\ref{th:CHL}.

Dunagan et al.~\cite{DST} (see also Spielman and Teng~\cite{ST:03})
performed a {\em smoothed analysis} of the running time of interior
point methods of linear programming by analyzing Renegar's condition number~$\CC'$, 
a variant of the GCC condition number $\CC$. 
Among other things, they proved the following:
$$
 \sup_{\bar{A}}\ \bE_{A\in N(\bar{A},\s^2)}(\ln\CC'(A)) = \Oh\Big(\ln\frac{mn}{\s}\Big).
$$
Here the supremum is over all
$\bar{A}\in\R^{n\times (m+1)}$ of
Frobenius norm at most one.
Our main result (Theorem~\ref{th:Emain}) shows that a similar bound actually holds for
a large class of random perturbation laws.

\subsection{A geometric approach to conic condition numbers}

B\"urgisser et al.~\cite{BCL:06a,BCL:08} recently extended
a result of Demmel~\cite{Demmel88} on conic condition numbers
from average-case analysis to smoothed analysis.
There, the perturbations of the inputs are modelled by uniform instead of Gaussian distributions.
This allows to perform the analysis in a general geometric framework that we explain next.

The set of ill-posed inputs to a computational problem is modelled as a lower dimensional
subset~$\Sigma$ of the data space~$\msD$, which is assumed to be furnished with a metric~$d$,
a distance function $\dist$ (not necessarily a metric),
and a volume measure.
In our cases of interest, $\msD$ is a Riemannian manifold, $d$ is the corresponding
metric, and $\dist = \sin(d)$.
The corresponding condition number~$\CC(a)$ of an input $a\in \msD$ is then defined as
\begin{equation}\label{eq:Cdef}
 \CC(a)=\frac1{\dist(a,\Sigma)} .
\end{equation}
This is an appropriate definition for many applications. In this model,
the set of inputs $a$ with condition $\CC(a) > \e^{-1}$ is given by the
$\e$-neighborhood
$$
T_\dist(\Sigma,\e)=\{a\in\msD\mid \dist(a,\Sigma) < \e\}.
$$
Let $B(\aa,\a) := \{a\in \msD \mid d(a,\aa)\le \a\}$
denote the ball centered at~$\aa\in \msD$ of radius~$\a$.
The task of a {\em uniform smoothed analysis} of~$\scC$
consists of providing good upper bounds on
$$
  \sup_{\aa\in \msD} \Prob_{a\in B(\aa,\a)} \{\scC(a) > \e^{-1}\},
$$
where $a$ is assumed to be chosen uniformly at random in $B(\aa,\a)$.
The probability occurring here thus has an immediate geometric meaning:
\begin{equation}\label{eq:relvol-gen}
 \Prob_{a\in B(\aa,\a)} \{\scC(a) > \e^{-1}\}
   = \frac{\vol\,(T_\dist(\Sigma,\e) \cap B(\aa,\a))}{\vol\,(B(\aa,\a))} .
\end{equation}
Thus uniform smoothed analysis means to provide
bounds on the volume of the intersection of
$\e$-neighborhoods of $\Sigma$, relative to the distance function~$\dist$,
with balls of radius $\a$.
We note that for compact data spaces~$\msD$, uniform smoothed analysis
interpolates transparently between worst-case
and average-case analysis.
Indeed, when $\a=0$ we get worst-case analysis,
while for $\a=\diam(\msD)$ we obtain average-case analysis.

For conic condition numbers, this general concept specializes in the following way.
The data space is the unit
$m$-sphere $S^m\subset\R^{m+1}$, $\Sigma$ is a lower dimensional
subset of $S^m$ such that $\Sigma=-\Sigma$
(in many applications it is an algebraic hypersurface)
and the conic condition number~$\scC(a)$ of $a\in S^m$ is defined by~(\ref{eq:Cdef})
with respect to the distance function $\dist(a,b):=\sin d(a,b)$, where~$d$
refers to the angular (i.e., Riemannian) distance in  $S^m$.
(We could as well consider the data space as the real projective space $\mathbb{P}^m$
on which $\dist$ defines a metric.)

For defining the GCC-condition number, one takes as the data space the
$n$th power $(S^m)^n= S^m\times\cdots\times S^m$ of the sphere~$S^m$
with the metric defined as
$d(A,B):=\max_{1\le i\le n} d(a_i,b_i)$,
where $A=(a_1,\ldots,a_n)$, $B=(b_1,\ldots,b_n)$
and $d(a_i,b_i)$ again denotes the angular distance in  $S^m$.
The distance function is defined by
$\dist(A,B):=\sin d(A,B)$ and the
GCC-condition number $\scC(A)$ is given by
$\scC(A)=1/\dist(A,\Sigma_{n,m})$, where
$\Sigma_{n,m}\subset (S^m)^n$
denotes the set of ill-posed inputs,
which is a semialgebraic subset of codimension one
(cf.~\S\ref{se:GCC} for details).

\subsection{Adversarial distributions} 

An advantage of the uniform model is that, for conic condition numbers,
results on smoothed analysis easily extend to more general families of
probability distributions. To obtain such robustness results
we rely on a general boosting technique developed in Hauser and M\"uller~\cite{HM:06}
and apply it similarly as in  Cucker et al.~\cite{CHL:09}

The framework is the following (cf. \S\ref{subsec:advers-distr}).
Consider the spherical cap
$B(\aa,\a) := \{x\in S^m\mid d(x,a) \le \a\}$
in the sphere~$S^m$ centered at~$\aa\in S^m$
and having angular radius~$\a\in (0,\pi/2]$.
Let $\nu$ denote the uniform measure on $B(\aa,\a)$ and suppose that
$\mu_{\aa}$ is a $\nu$-absolutely continuous probability measure on $S^m$, say
$\mu_{\aa}(G) = \int_G f d\,\nu$
for measurable sets~$G$. We further assume that the density $f$ is
of the form $f(x) = g(\sin d(x,\aa))$
with a monotonically decreasing function
$g\colon [0,\s]\to [0,\infty]$ of the form
$$
 g(r) =   C r^{-\beta}\, h(r),
$$
where $\s:=\sin\a$, $\beta < m$,
$h\colon [0,\s]\to [0,\infty)$ is a continuous function satisfying $h(0)\ne 0$,
and $C=I_m(\a)/I_{m-\beta}(\a)$ is a normalizing constant, 
where $I_k(\a) := \int_0^\a (\sin t)^{k-1}\, dt$,
cf.~\S\ref{se:dist}.
Thus, the support of $\mu_{\aa}$ is contained in $B(\aa,\a)$
and the density of $\mu_{\aa}$ is radially symmetric with a pole of order~$\beta$
at $\aa$.
Such distributions were called {\em adversarial} in~\cite{CHL:09}.
The exponent~$\beta$ and
the quantity $H:=\sup_{0\le r \le \s} h(r)$ are the only parameters entering
the bound below.
In the special case $\b=0$,
the density of $\mu_{\aa}$ does not have a singularity
and the situation considerably simplifies:
We have $C=1$, $g=h$ and the estimate $\mu_{\bar{a}}(B)\leq H\cdot \nu(B)$
holds for any measurable set~$G$.

To extend this estimate
to the general case, the smoothness parameter~$s$
of a $\nu$-absolutely continuous distributions
was defined in~\cite{HM:06}, and in~\cite[Lemma~3.2]{CHL:09}, 
it was shown that $s=1-\beta/m$.
This means that $s$ is
the largest number $s' > 0$ for which the following is true:
For every $\veps>0$ there exists $\delta(\veps)>0$
such that $\nu(G)\leq\delta(\veps)$ implies $\mu_{\bar{a}}(G)\leq\nu(B)^{s'-\veps}$
for all measurable sets $G$.
This allows to obtain tail bounds for $\mu_{\aa}$ from tail bounds for $\nu$.

B\"urgisser et al.~\cite{BCL:08} provided a general smoothed analysis of conic
condition numbers for uniform perturbations.
This was recently extended by Cucker et al.~\cite{CHL:09}
to the model of adversarial perturbations, who
obtained the following robust smoothed analysis estimate.

\begin{theorem}\label{th:CHL}
Let $\scC$ be a conic condition number with set of ill-posed inputs
$\Sigma\subseteq S^m$ and assume that $\Sigma$ is contained in an
algebraic hypersurface of degree~$d$.
Then we have for a random perturbation from any
adversarial distribution~$\mu_{\aa}$ on $B(\aa,\a)$
with center $\aa\in S^m$, angular radius $\a\in (0,\pi/2]$
and parameters $\b,H,\s =\sin\a$ that
$$
 \bE_{\mu_{\aa}}( \ln\scC)\ \le\
  \ln\frac{m^2d}{\s} + \frac1{1-\beta/m} \ln\frac{2eH^2 m}{\ln(\pi m/2)}
    + \ln\frac{13\pi}{2} .
$$
\end{theorem}

It is remarkable that the only problem dependent parameters entering the above bound
are the dimension~$m$ and the degree~$d$.
The only distribution dependent parameters entering the bound are
$\sigma$, $\beta$, and $H$.
This result has a wide range of applications to linear and polynomial equation solving.

\subsection{Main results}

The goal of this paper is to prove the following analogue of Theorem~\ref{th:CHL}
for the GCC-condition number of the linear programming feasibility problem.

\begin{theorem}\label{th:Emain} 
Let $\scC$ denote the GCC condition number defined on $(S^m)^n$ and $n>m+1$.
Suppose that $a_i$ is randomly chosen from an adversarial distribution~$\mu_{\aa_i}$
on $B(\aa_i,\a)$ with center $\aa_i\in S^m$, angular radius $\a\in (0,\pi/2]$,
and parameters $\b,H,\s=\sin\a$, independently for $i=1,\ldots,n$.
Then the random matrix $A$ with rows $a_i$ satisfies
\begin{eqnarray*}
 \bE \big(\ln \CC(A)\big) & = &
 \Oh \Big( {\textstyle\big(1-\frac{\b}{m}\big)^{-2}}\cdot \ln \frac{nH}{\s}\Big).
\end{eqnarray*}
In the case $\beta=0$, we have more precisely,
$$
 \bE \big(\ln \CC(A)\big) \le \textstyle 12 \ln n + 17 \ln m
       + 6 \ln \frac{1}{\s} + 8 \ln H + 29 =  \Oh\big(\ln\frac{nH}{\s}\big) .
$$
\end{theorem}

As an application
let $T(A)$ denote the number of iterations of the primal-dual interior point method
of~\cite{cupe:02} for solving the linear programming feasibility problem
$\exists x\in\R^{m+1}\ Ax < 0$ (or its dual problem).
Theorem~\ref{th:Emain} implies that
$$
 \bE (T(A)) = \Oh\big({\textstyle\big(1-\frac{\b}{m}\big)^{-2}}\cdot\sqrt{n}\cdot\ln\frac{n H}{\s}\big) ,
$$
for a random matrix $A$ with independent rows $a_i$
from a adversarial distribution $\mu_{\aa_i}$
with parameters $\b,H$.

For the uniform distribution on $S^m$, by essentially the same method,
we can improve the estimates of Theorem~\ref{th:Emain}
obtaining a result that was previously shown in \cite{BCL:08}
by a very different technique.

\begin{corollary}\label{cor:average}
Suppose that the rows of the matrix $A$  are independently chosen in $S^m$
according to the uniform distribution and $n>m+1$.
Then
$$
 \bE \big(\ln \CC(A)\big) = \Oh\big(\ln m \big).
$$
\end{corollary}

The paper is organized as follows.
Section~\ref{se:prelim} is devoted to preliminaries on
spherically convex sets, the GCC-condition number,
and adversarial probability distributions.
In Section~\ref{se:USA} we state and prove probability tail bounds
for the GCC-condition number under adversarial random perturbations
and prove our main results.
The proof essentially reduces the adversarial to the uniform case
and then uses geometric arguments for the uniform case.
A principal ingredient of the proof of the uniform case
is an upper bound on the volume of
the neighborhood of spherically convex sets
(Theorem~\ref{th:volbdconv}), that is stated in
Section~\ref{se:volbound}, but whose proof is deferred
to Section~\ref{se:volnbS}.
The proof of the latter result proceeds along the lines of~\cite{BCL:08}
and uses some deeper results from integral and differential geometry
(Weyl's tube formula, kinematic formula).

\bigskip

\noindent
{\bf Acknowledgments.}
We thank Felipe Cucker and Martin Lotz for numerous helpful discussions.
We are grateful to an anonymous referee for constructive criticism
that led to more general results and a significantly better presentation of the paper.

\section{Preliminaries}\label{se:prelim}

\subsection{Convex sets in spheres}\label{se:conv}

A general reference about convex sets is \cite{webs:94}.
Glasauer's thesis~\cite{glas:95} is
a useful reference for the integral geometry of
spherically convex sets.

A {\em convex cone} in $\R^{m+1}$ is a subset
that is closed under addition and multiplication with nonnegative scalars.
We denote by $\cone(M)$ the convex cone generated by a subset
$M\subseteq \R^{m+1}$. More specifically,
the convex cone generated by points $a_1,\ldots,a_k\in \R^{m+1}$ is given by
$$
\mbox{$\cone\{a_1,\ldots,a_k\}:=\{x\in\R^{m+1}\mid
    \exists\lambda_1\ge 0,\ldots,\lambda_k \ge 0\quad x=\sum_{i=1}^k\lambda_i a_i \}$.}
$$
A convex cone~$C$ is called {\em pointed} iff $C\cap(-C)=\{0\}$. It
is known that $C$ is pointed iff $C\setminus\{0\}$
is contained in an open halfspace whose bounding hyperplane goes
through the origin. Clearly,
if $a_i\neq0$ for $i=1,\ldots,k$, then
$\cone\{a_1,\ldots,a_k\}$ is pointed
iff $0$ is not contained in the convex hull
$\conv\{a_1,\ldots,a_k\}$.

We use convex cones to define the notion of convexity for subsets of the sphere
$S^m=\{x\in \R^{m+1}\mid \|x\| =1\}$.
Let $x,y\in S^m$ be such that $x\ne\pm y$. We call
$[x,y] := \cone\{x,y\}\cap S^m$ the {\em great circle segment} connecting $x$ and~$y$.

\begin{definition}\label{def:convex}
A subset $K$ of $S^m$ is called {\em(spherically) convex} iff
we have $[x,y]\subseteq K$ for all $x,y\in K$ with $x\ne\pm y$.
We call $K$ {\em properly convex} iff it is nonempty, convex,
and does not contain a pair of antipodal points.
\end{definition}

We denote by $\scone(M):=\cone(M)\cap S^m$
the {\em convex hull} of a subset $M$ of $S^m$,
which is the smallest convex set containing $M$.
Clearly, $M$ is convex iff $M=\scone(M)$.
The closure of a convex set is convex as well.
It is easy to see that a convex subset $K$ of $S^m$ is contained in
a closed halfspace, unless $K=S^m$.
A properly convex set $K$ is always contained in an open halfspace.

\begin{definition}\label{def:dual}
The {\em dual set} of a subset $M\subseteq S^m$ is defined as
$$
 \breve{M} := \{ a\in S^m\mid \forall x\in M\ \<a,x\>\le 0 \}.
$$
\end{definition}

Clearly, $\breve{M}$ is a closed convex set disjoint to $M$.
The hyperplane separation theorem implies that the dual of $\breve{M}$
equals the closure of $\scone(M)$.
We note that
$M\subseteq N$ implies $\breve{M}\supseteq\breve{N}$.
Finally, it is important that $\breve{M}$ has nonempty interior iff $M$
does not contain a pair of antipodal points, that is,
``nonempty interior'' and ``properly convex'' are dual properties.

By a {\em convex body} $K$ in $S^m$ we will understand a closed
convex set $K$ such that both $K$ and $\breve{K}$ have nonempty interior,
i.e., both are properly convex.
The map $K\mapsto\breve{K}$ is an involution of the set of convex bodies in $S^m$.

\subsection{Distances, neighborhoods, and volumes}\label{se:dist}

We denote by $d(a,b)\in [0,\pi]$ the angular distance between
two points $a,b$ on the sphere~$S^m$.
Clearly, this defines a metric on $S^m$.
The (closed) {\em ball} of radius $\a\in[0,\pi]$ around $a\in S^m$ is defined as
$$
B(a,\a):=\{x\in S^m\mid d(x,a) \le \a\} = \{x\in S^m \:|\: \lan a,x\ran \ge \cos\a \}.
$$
This is the same as the {\em spherical cap} with center~$a$ and angular radius $\alpha$.
$B(a,\a)$ is convex iff $\a\le \pi/2$ or $\a=\pi$.
In the case $\a\le \pi/2$, the dual set of $B(a,\a)$ equals $B(-a,\pi/2-\a)$.

For a nonempty subset $M$ of $S^m$ we define the distance
of $a\in S^m$ to~$M$ as
$d(a,M):=\inf\{d(a,x)\mid x\in M\}$.
The dual set of $M$ can be characterized
in terms of distances by:
$a\in\breve{M} \Longleftrightarrow d(a,M)\ge \pi/2$.

\begin{lemma}\label{le:dkk}
Let $K$ be a convex body in $S^m$ and
$a\in S^m\setminus (K\cup\breve{K})$.
Then $d(a,K) + d(a,\breve{K}) = \pi/2$.
\end{lemma}

\begin{proof}
Let $b\in K$ such that $\ph:=d(a,b)=d(a,K)$.
Since $a\not\in\breve{K}$ we have $\ph<\pi/2$.
The point $b^\ast:=\<a,b\>\,b$ is therefore nonzero and contained
in $C:=\cone(K)$.
Put $p^\ast:=a-b^\ast$.
Then $\<p^\ast,b \>=0$,
$\<p^\ast,a\>=\sin^2\ph$, and
$\<p^\ast,p^\ast\>=\sin^2\ph$.
In particular $p^\ast\ne 0$.

By construction, $b^\ast$ is the point of $C$ closest to $a$.
It follows that
$\{x\in\R^{m+1} \mid \<p^\ast,x\> =0\}$
is a supporting hyperplane of $C$.
Hence $\<p^\ast,x \>\le 0$ for all $x\in C$ and the point
$p:=p^\ast/\|p^\ast\|$ therefore belongs to $\breve{K}$.
Moreover, $\<p,a\>=\sin\ph$,
which implies $d(a,p)=\pi/2-\ph$.
Hence
$$
 d(a,K) + d(a,\breve{K})\le d(a,b) + d(a,p) = \pi/2.
$$
To complete the proof it suffices to show that
$d(a,\breve{K})=d(a,p)$.
Suppose there exists $p'\in\breve{K}$ such that $d(a,p') < d(a,p)$.
Then
$d(b,p') \le d(b,a) + d(a,p') < d(b,a) + d(a,p) =\pi/2$
which contradicts the fact that $b\in\breve{K}$.
\end{proof}

Sometimes it will be useful to consider the
{\em projective distance} between two points $a,b\in S^m$,
which is defined as $\dP(a,b):=\sin d(a,b)$.
It is straightforward to check that $\dP$ satisfies the triangle
inequality. However, it is not a metric on $S^m$, as
$\dP(a,b)=0$ iff $a=\pm b$. Hence the
ball of radius $\sin\a$,
measured with respect to the projective distance,
equals $B(a,\a)\cup B(-a,\a)$.
We denote this set suggestively by $B(\pm a,\a)$
and call it the {\em projective ball} with center $\pm a$
and radius $\a$.

For $0\le\ph\le \pi/2$, the {\em $\ph$-neighborhood} of a
nonempty subset $M$ of $S^m$ is defined as
$T(M,\ph) := \{x\in S^m\mid d(x,M) < \ph  \}$.
If $M$ is the boundary $\partial K$ of a properly convex set $K$ in $S^m$,
we call
$$
 T_o(\partial K,\ph) := T(\partial K,\ph)\setminus K
 \quad\mbox{and}\quad
 T_i(\partial K,\ph) := T(\partial K,\ph) \cap K
$$
the {\em outer $\ph$-neighborhood} and
{\em inner $\ph$-neighborhood} of $\partial K$,
respectively.
Clearly, we have
$T(\partial K,\ph)= T_o(\partial K,\ph)\cup T_i(\partial K,\ph)$.

In order to compute the $m$-dimensional volume of such neighborhoods,
the following functions $J_{m,k}(\a)$ are relevant:
\begin{equation}\label{eq:def-J-fcts}
  J_{m,k}(\a) := \int_0^\a (\sin\rho)^{k-1}\, (\cos\rho)^{m-k}\, d\rho \quad\mbox{($1\le k\le m$)}.
\end{equation}
Recall that
$\Oh_m := \vol S^m = 2\pi^{(m+1)/2}/\Gamma((m+1)/2)$
equals the $m$-dimensional volume of~$S^m$.
It is known that
$\vol\, T(S^{m-k},\a) = \Oh_{m-k}\Oh_{k-1} J_{m,k}(\a)$.
Some estimations of these volumes can be found in~\cite[Lemmas~2.1-2.2]{BCL:08}.

\subsection{The GCC condition number}\label{se:GCC}

We study the problem of deciding for a given instance $A\in\R^{n\times (m+1)}$
whether there exists a nonzero solution
$x\in\R^{m+1}\setminus\{0\}$ such that $Ax\le 0$.
In the following we assume that $n>m+1$.
(The other case is considerably less interesting.)
Without loss of generality we may assume that the row vectors $a_i$ have euclidean length one,
and hence interpret $A=(a_1,\ldots,a_n)$ as an element of the product $(S^m)^n$ of spheres.

We write $\scone(A):=\scone\{a_1,\ldots,a_n\}$ for the convex hull of the given points.
The set of solutions in $S^m$ of the system of inequalities $Ax\le 0$
equals the dual set of $\scone(A)$.

\begin{definition}\label{def:FIS}
An instance $A\in (S^m)^n$ is called \emph{feasible} iff its set of solutions is nonempty,
otherwise $A$ is called \emph{infeasible}.
An instance $A$ is called \emph{strictly feasible} iff its set of solutions has nonempty interior.
We denote by $\Fe_{n,m}$ and  $\Fe_{n,m}^\circ$ the set of feasible and
strictly feasible instances, respectively.
The set of \emph{ill-posed instances} is defined as
$\Sigma_{n,m}:=\Fe_{n,m}\setminus\Fe_{n,m}^\circ$.
The set of infeasible instances is denoted by $\In_{n,m}$.
\end{definition}

\begin{remark}
An instance $A\in (S^m)^n$ is strictly feasible iff
$\scone(A)$ is properly convex, 
that is, $\cone(A)$ is pointed.
Furthermore, a feasible instance~$A$ is ill-posed iff
$0$ is contained in the euclidean convex hull of $a_1,\ldots,a_n$
(cf.~\cite[Lemma~3.2]{BCL:08a}).
\end{remark}

We remark that $\Fe_{n,m}$ is a
compact subset of $(S^m)^n$
with nonempty interior $\Fe_{n,m}^\circ$ and
topological boundary~$\Sigma_{n,m}$.
Moreover, $\In_{n,m}$ is nonempty and $\Sigma_{n,m}$ is also
the topological boundary of $\In_{n,m}$.
(Here we use $n>m+1$.)

We define a metric on $(S^m)^n$ by setting for $A,B\in (S^m)^n$
with components $a_i,b_i\in S^m$
$$
d(A,B) \; := \; \max_{1\le i\le n} d(a_i,b_i) \; .
$$
The distance of $A$ to a nonempty subset $M\subseteq(S^m)^n$ is defined as
$d(A,M) := \inf\{d(A,B) \:|\: B\in M\}$.
We denote by $B(\bar{A},\a):=\{A\in (S^m)^n\mid d(A,\bar{A}) \le \a \}$ the closed ball
with center $\bar{A}$ and radius~$\a$.
Clearly, this is the product of the balls $B(\bar{a}_i,\a)$ for $i=1,\ldots,n$.

The following definition is due to Goffin~\cite{goff:80}
and Cheung and Cucker~\cite{ChC:01}.

\begin{definition}\label{def:GCC}
The GCC condition number of $A\in (S^m)^n$ is
defined as $\scC(A) = 1/\sin d(A,\Sigma_{n,m})$.
\end{definition}

This condition number can be characterized in a more explicit way.

\begin{definition}\label{def:SIC}
A \emph{smallest including cap} (SIC) for $A\in (S^m)^n$ is a
spherical cap of minimal radius containing the points $a_1,\ldots,a_n$.
\end{definition}

We remark that by a compactness argument, a SIC always exists.
It can be shown that a SIC is unique if $A$ is strictly feasible.
However, for infeasible~$A$, there may be several SICs (consider for instance three
equidistant points on the circle).
We denote the radius of a SIC of $A$ by $\rho(A)$.
An instance $A$ is strictly feasible iff $\rho(A)<\pi/2$.
For more information on this we refer to~\cite{ChCH:05,BCL:08a}.

The following result is due to Cheung and Cucker~\cite{ChC:01}.
This characterization is essential for any probabilistic analysis of the GCC condition number.

\begin{theorem}\label{thm:GCC-dist}
We have
  \[ d(A,\Sigma_{n,m}) \; = \;
  \begin{cases} \frac{\pi}{2}-\rho(A) & \text{if $A\in\Fe_{n,m}$,}
             \\ \rho(A)-\frac{\pi}{2} & \text{if $A\in(S^m)^n\setminus\Fe_{n,m}$.}
  \end{cases} \; \]
In particular, $d(A,\Sigma_{n,m})\le\frac{\pi}{2}$ and
$\scC(A)^{-1} = |\cos \rho(A)|$.
\end{theorem}

\subsection{Adversarial probability distributions}\label{subsec:advers-distr}

The proof of our robustness result relies on a general boosting technique
developed in~\cite{HM:06}, that allows to extend probability tail bounds obtained
for one fixed distribution to larger classes of distributions.
We explain this technique in our situation of interest.

Let $\nu:=\nu_{\aa,\s}$ denote the uniform distribution on the spherical disk~$B(\aa,\a)$,
where $\a\in(0,\frac{\pi}{2}]$ and $\s:=\sin\a$.
We assume that $\mu$ is a $\nu$-absolutely continuous probability measure, i.e.,
it can be written with a density $f$ as $\mu(G) = \int_G f d\,\nu$
for Borel measurable sets~$G$.
In certain cases, it is possible to bound $\mu(G)$ in terms of $\nu(G)$
if the latter is sufficiently small. This is done with the
{\em smoothness parameter} $\sp$ of $\mu$ with respect to $\nu$, which is defined as
$\sp := \lim_{\d\to 0} \inf\d$, where we have set for $\d\in (0,1)$,
using the convention $\ln 0:= -\infty$,
$$
 \inf\d := \inf\Big\{\frac{\ln \mu(G)}{\ln \nu(G)}\ \mid\ 0 < \nu (G) \le \d \Big\} .
$$
If $\sp$ is positive, we say that $\mu$ is {\em uniformly $\nu$-absolutely continuous}.
In this case, it is easy to see that $\sp$ is the largest nonnegative real number~$s'$
with the property that for all~$\e>0$ there exists a tolerance $\d(\e)>0$ such that
$\nu(G) \le \d(\e)$ implies $\mu(G) \le \nu(G)^{s'-\e}$ for all~$G$, cf.~\cite{HM:06}.

We will apply this framework to a specific class of distributions~$\mu$.
Not only will it be important to know the smoothness parameter, but also to
explicitly compute bounds for the tolerance $\d(\e)$.

An {\em adversarial probability distribution}~$\mu_{\aa}$, for $\aa\in S^m$,
was defined in~\cite{CHL:09} as a $\nu$-absolutely continuous measure
given by $\mu_{\aa}(G) = \int_G f d\,\nu$, where again $\nu$ denotes the uniform
distribution on $B(\aa,\a)$. 
We further require that the density $f$ is of the form $f(x) = g(\sin d(x,\aa))$ with a
monotonically decreasing function $g\colon [0,\s]\to [0,\infty]$ given by
$$
 g(r) =   C r^{-\beta}\, h(r),
$$
where $0\le\beta < m$, and
$h\colon [0,\s]\to [0,\infty)$ is a continuous function satisfying $h(0)\ne 0$.
Thus, the support of $\mu_{\aa}$ is contained in $B(\aa,\a)$
and the density of $\mu_{\aa}$ is radially symmetric with a pole
of order $\b$ at $\aa$. Clearly, $\aa$ is the mean of $a$ with
respect to $\mu_{\aa}$.
It is convenient to assume the normalization
$C := I_m(\a)/I_{m-\beta}(\a)$,
where
$$
I_k(\a) := J_{k,k}(\a) = \int_0^\a (\sin t)^{k-1}\, dt,
$$
compare~\eqref{eq:def-J-fcts}
(we slightly deviate here from the notation in~\cite{CHL:09}).
Then $\mu_{\aa}$ is a probability distribution iff
$$
 I_{m-\b}(\a) = \int_0^\a (\sin t)^{m-\b-1}\,h(\sin t)\,dt .
$$
The maximum of $h$ is denoted by $H:=\sup_{0\le r\le\s} h(r)$,
which is easily seen to satisfy $H\geq 1$.
We note that the uniform distribution on $B(\aa,\a)$
is obtained by choosing $\b=0$, $C=1$ and for $g=h$ the function identically
equal to~$1$. 

The following technical lemma is an immediate consequence of
\cite[Lemmas~3.2-3.3, Equation~(3.1)]{CHL:09}.

\begin{lemma}\label{lem:CHL:09}
\begin{enumerate}
\item The smoothness parameter of $\mu$ with respect to $\nu$ equals
$\sp = 1-\b/m$.

\item For $c:=\frac12 s = \frac{1}{2}(1-\frac{\beta}{m})\leq\frac{1}{2}$
and the {\em tolerance}~$\d_c$ defined by
$$
   \d_c \ :=\ \frac{2}{\pi m}\left(\frac{1}{H}\cdot \sqrt{1-\left(\frac{2}{\pi m}
   \right)^{\frac{1}{m}}}\right)^{1/c}
$$
we have for all measurable $G\subseteq S^m$, 
$$
 \nu(G)\le\d_c \ \Longrightarrow\  \mu_{\aa}(G)\le (\nu(G))^c.
$$
\end{enumerate}
\end{lemma}

\begin{remark}
In the case $\b=0$ where the density of $\mu_{\aa}$ has
no singularity the situation simplifies.
Clearly, $\mu_{\aa}(G) \le H\nu (G)$ for all~$G$.
This directly implies that the smoothness parameter equals~$1$.
Moreover, the tolerance $\d = \frac1{H^2}$ is sufficient for
the implication $\nu(G) \le\d \Rightarrow \mu_{\aa}(G) \le \nu(G)^\frac12$.
\end{remark}

Since we will work in the product of spheres $(S^m)^n$ we define
on it the {\em adversarial distributions} $\mu_{\bar{A}}$
with center $\bar{A}=(\aa_1,\ldots,\aa_n)\in(S^m)^n$
by taking the product measure
$\mu_{\bar{A}} := \mu_{\aa_1}\times\ldots\times\mu_{\aa_n}$.
This is a probability distribution whose support is contained in the product of caps
$B(\bar{A},\a):=B(\aa_1,\a)\times\ldots\times B(\aa_n,\a)$.

\section{Robust smoothed analysis of $\CC(A)$}\label{se:USA}

The goal of this section is to provide smoothed analysis estimates
for the condition number $\CC(A)$ in the model where
$A=(a_1,\ldots,a_n)$ is chosen at random according
to an adversarial distribution~$\mu_{\bar{A}}$.
The center $\bar{A}\in(S^m)^n$ of the perturbation is arbitrary.
Recall that $\Fe_{n,m}$ and
$\In_{n,m}$ denote the sets of feasible and infeasible instances
$A\in (S^m)^n$, respectively.

\begin{theorem}\label{th:main} 
Let $\a\in (0,\pi/2]$, $\s=\sin\a$, and
$\bar{A}\in (S^m)^n$, where \mbox{$n>m+1$}.
Assume that
$A\in B(\bar{A},\a)$ is chosen at random
according to an adversarial distribution $\mu_{\bar{A}}$.
Let the exponent $c=\frac12(1-\b/m)$ and the tolerance $\d_c$
be defined as in Lemma~\ref{lem:CHL:09}.
Then we have
\begin{equation}\label{eq:F}\tag{F}
 \Prob\{ A\in \Fe_{n,m}\mbox{ and } \CC(A) \ge t\}
 \le n\left(\frac{13m(m+1)}{2\s}\right)^c\, t^{-c} .
\end{equation}
provided $t\ge \frac{13m(m+1)}{2\s\d_c}$.
Moreover, we have for $t\ge 1$,
\begin{equation}\label{eq:I}\tag{I}
 \Prob\{ A\in \In_{n,m} \mbox{ and } \CC(A) \ge t\}
   \le n \left(\frac{1690 m^2(m+1)}{4\s^2}\right)^c t^{-c} (\d_c^{-c} + c n \ln t) .
\end{equation}
\end{theorem}

\begin{remark}
For the uniform distribution ($\b=0$, $h\equiv1$) the bounds in~\eqref{eq:F}
and~\eqref{eq:I} can easily be improved by avoiding the use of
Lemma~\ref{lem:CHL:09}. In particular, on the right-hand side
of~\eqref{eq:F} and~\eqref{eq:I} one gets $t^{-1}$ and $t^{-1}\ln t$ instead
of $t^{-1/2}$ and $t^{-1/2}\ln t$.
\end{remark}

The overall strategy of the proof of Theorem~\ref{th:main} is the
same as in~\cite{DST}.
However, the crucial component in~\cite{DST},
namely a result due to Ball~\cite{ball:93},
is substituted by Corollary~\ref{cor:volbdconv},
which is stated in the next section.

\subsection{A bound on the volume of convex neighborhoods}\label{se:volbound}

We state here an upper bound on the volume of the intersection of
neighborhoods of spherically convex sets with spherical caps.

\begin{theorem}\label{th:volbdconv}
Let $K$  be a properly convex subset of $S^{m}$, let
$a\in S^m$, and $0<\a,\ph\le\pi/2$.
Then, writing $\s=\sin\a$ and $\e=\sin\ph$, we have the
following upper bound for the volume of the outer neighborhood
of~$\partial K$:
$$
 \frac{\vol (T_o(\partial K,\ph) \cap B(\pm a,\a))}{\vol B(\pm a,\a)} \le
   \sum_{k=1}^{m-1} {m\choose k}\, \Big(1 + \frac{\e}{\s}\Big)^{m-k}\, \Big(\frac{\e}{\s}\Big)^k
    + \frac{m\Oh_m}{2\Oh_{m-1}}\, \Big(\frac{\e}{\s}\Big)^m .
$$
The same upper bound holds for the volume
$\frac{\vol (T_i(\partial K,\ph) \cap B(\pm a,\a))}{\vol B(\pm a,\a)}$
of the inner neighborhood of $\partial K$.
\end{theorem}

The main result of \cite[Theorem 1.2]{BCL:08}
gives a bound on the volume of the intersection of a projective ball
$B(\pm a,\a)=B(a,\a)\cup B(-a,\a)$
with the $\ph$-neighborhood of
a real algebraic hypersurface in the sphere $S^m$,
given as the zero set of polynomial of degree $d$.
The above theorem says that essentially the same volume bound holds
for the boundary of a properly convex set~$K$ in $S^m$,
if  we formally replace in this bound the degree~$d$ by $1/2$.

The proof of Theorem~\ref{th:volbdconv}, which is quite involved,
is deferred to Section~\ref{se:volnbS}.
By essentially the same argument as in the proof of \cite[Prop. 3.5]{BCL:08}
one can derive from Theorem~\ref{th:volbdconv} the following corollary.

\begin{corollary}\label{cor:volbdconv}
Under the assumptions of Theorem~\ref{th:volbdconv} we have
the following upper bound for the volume of the outer neighborhood of $\partial K$:
$$
  \frac{\vol (T_o(\partial K,\ph) \cap B(\pm a,\a))}{\vol B(\pm a,\a)} \le
  \frac{13m}{4}\, \frac{\e}{\s}\quad \mbox{ if\/ $\e\le\frac{\s}{2m}$.}
$$
The same upper bound holds for the relative volume
of the inner neighborhood of $\partial K$.
\end{corollary}

\subsection{Two auxiliary results}

The proofs of the following two results are similar as in Dunagan et al.~\cite{DST}.
We use the notation $[n]:=\{1,2,\ldots,n\}$ for $n\in\N$.

\begin{proposition}\label{pro:AF}
Let $A=(a_1,\ldots,a_n)\in\Fe_{n,m}^\circ$,
$0<\ph\le \pi/2$, and $\e=\sin\ph$.
If $\CC(A)\ge (m+1)\,\e^{-1}$,
then there exists $i\in [n]$ such that
$$
 a_i \in T(\partial K_i,\ph)\setminus K_i,
$$
where $K_i:=-\scone\{a_1,\ldots,a_{i-1},a_{i+1},\ldots,a_n\}$.
\end{proposition}

\begin{proof}
There exists $q\in\scone(A)$ such that $\<a_i,q\> >0$ for all $i\in
[n]$. Indeed, if $q$ is taken as the center of the SIC of $A$ then
this follows from~\cite[Lemma~4.5]{ChCH:05} (see also
\cite[Lemma~3.2]{BCL:08a}).

We note that $a_i\not\in K_i$ for all $i\in[n]$. Otherwise
$0\in\conv\{a_1,\ldots,a_n\}$, hence $A\in\Sigma_{n,m}$, which
contradicts our assumption that $A$ is strictly feasible. It follows
that $d(a_i,\partial K_i) =d(a_i,K_i)$.

We assume now $d(a_i,K_i) > \ph$ for all $i\in [n]$. Our goal is to
show that $\sin d(A,\Sigma_{n,m}) > \frac1{m+1}\e$. Then we are
done, since $\CC(A)^{-1}=\sin d(A,\Sigma_{n,m})$ by
Definition~\ref{def:GCC}.

We proceed now similarly as in~\cite[Lemma~2.3.10]{DST}.
By continuity we assume w.l.o.g.\ that $\ph<\pi/2$. We distinguish two
cases. If $a_i\not\in\breve{K_i}$, then Lemma~\ref{le:dkk} tells us
that $d(a_i,K_i) + d(a_i,\breve{K_i}) = \pi/2$. Hence
$d(a_i,\breve{K_i}) < \pi/2-\ph$. Choose $p_i\in\inte(\breve{K_i})$
such that $d(a_i,p_i) < \pi/2-\ph$. This implies $\<a_i,p_i\> >
\cos(\pi/2-\ph) = \e$. In the case $a_i\in\breve{K_i}$ we take any
$p_i\in\inte(\breve{K_i})$ close enough to $a_i$ such that $
\<a_i,p_i\> > \e$.

In both cases we have achieved the following
\begin{equation}\label{eq:ine}
 \<a_i,p_i\> > \e \quad\mbox{and}\quad
  \forall j\ne i\ \ \<a_j,p_i\> > 0 .
\end{equation}
This implies for all~$i$ that $\<p_i,q\> >0$, as
$q\in\cone\{a_1,\ldots,a_n\}$.

Consider now for $i\in [n]$ the following convex sets in $S^m$
$$
 C_i := \{ x\in S^{m} \mid \<a_i,x\> > \frac{\e}{m+1} \mbox{ and } \<x,q\> > 0\}
$$
containing $p_i$. We claim that the intersection of any $m+1$ of
these sets is nonempty. Indeed, let $I\subseteq [n]$ be of
cardinality $m+1$ and consider $p^*:=\frac1{m+1}\sum_{j\in I} p_j$.
Note that $\|p^*\| \le 1$. We obtain for any $i\in I$, using
(\ref{eq:ine}),
$$
 \<a_i,p^*\> = \frac1{m+1}\sum_{j\in I} \<a_i,p_j\>
  \ge \frac1{m+1} \<a_i,p_i\> > \frac{\e}{m+1} .
$$
Moreover, $\<p^*,q\> >0$, hence $p^*\ne 0$. It follows that
$p:=p^*/\|p^*\|$ is contained in $ C_i$ for any $i\in I$, which
shows the claim.

Consider the affine hyperplane $E:=\{x\in\R^{m+1} \mid \<x,q\> =
1\}$ of dimension $m$ and the perspective map
$$
\pi\colon\{x\in S^{m} \mid \<x,q\> > 0\} \to E, x\mapsto
\<q,x\>^{-1} x .
$$
Then the $\pi(C_i)$ are convex subsets of $E$, with the property
that any $m+1$ of these have a nonempty intersection. Helly's
theorem~\cite{webs:94} implies that $\pi(C_1)\cap\cdots\cap\pi(C_n)$
is nonempty. Hence there is a point $a\in \bigcap_{i=1}^n C_i$. We
have $d(a_i,a) < \a :=\arccos ((m+1)^{-1}\e)$ for all $i\in [n]$.
Hence the spherical cap $B(a,\a)$ strictly contains all
$a_i$. The radius $\rho(A)$ of the SIC of~$A$ is therefore strictly
smaller than $\a$. Hence, by Theorem~\ref{thm:GCC-dist}, $\sin
d(A,\Sigma_{n,m}) = \cos\rho(A) > \cos\a = (m+1)^{-1}\e$, as
claimed.
\end{proof}

The next proposition on the transition from the feasible to the
infeasible case is similar as \cite[Lemma~2.3.14]{DST}.

\begin{proposition}\label{pro:IF}
Let $A=(a_1,\ldots,a_n)\in\Fe_{n,m}$ and $K:=-\scone(A)$. If $b\in
K$, then $(A,b):=(a_1,\ldots,a_n,b)$ is infeasible or ill-posed and
we have
$$
 \CC(A,b)\, \sin d(b,\partial K) \le 10\, \CC(A) .
$$
\end{proposition}

\begin{proof}
W.l.o.g.\ $A$ is strictly feasible. The set of solutions
$$
 C:=\{x\in S^m \mid \<a_1,x \>\le 0,\ldots,\<a_n,x \>\le 0 \}
$$
is the dual of $\scone(A)$. Hence $\breve{C}=\scone(A)$. This means
that $a\in K$ iff $\<a,x\> \ge 0$ for all $x \in C$. Therefore, we
have for all $a\in S^m$,
\begin{equation}\label{eq:Ksif}
 a\not\in K \Longleftrightarrow
 \exists x \in C\ \<a,x\> < 0
 \Longleftrightarrow \mbox{$(a_1,\ldots,a_n,a)$ is strictly feasible},
\end{equation}
where the second equivalence follows from the assumption that $C$
has non\-empty interior. A similar argument shows that
$(a_1,\ldots,a_n,a)$ is ill-posed iff $a \in\partial K$. Therefore,
we have
$$
 d(b,\partial K) = \min\{ d(b,a) \mid a\in S^m\mbox{  such that }
   (a_1,\ldots,a_n,a) \in \Sigma_{n+1,m} \} .
$$

For proving the proposition we can assume without loss of generality
that $b\in K\setminus\partial K$. Then $(A,b)$ is not strictly
feasible by~(\ref{eq:Ksif}). Moreover, since $b\not\in\partial K$,
$(A,b)$ is not ill-posed. Hence $(A,b)$ is infeasible. We put now
$\dnew  := \sin d(b,\partial K)$ and claim that
\begin{equation}\label{eq:dpK}
 \dnew \le  \min_{x\in C}\<b,x\> .
\end{equation}
In order to show this, suppose $q\in C$. The
equivalence~(\ref{eq:Ksif}) and $b\in K$ imply that $\cos\theta :=
\<b,q\> \ge 0$. W.l.o.g. we may assume that
$\|b-q\cos\theta\|^2=1-\cos^2\theta$ is positive (otherwise
$\theta=0$, $b=q$, and $\<b,q\>=1\ge \dnew$). It therefore makes sense
to define $b':=(b-q\cos\theta)/\|b-q\cos\theta\|$. Then $b'\in S^m$
and $\<b',q\>=0$. Note that $d(b,b')=\pi/2-\theta$. Therefore
$(a_1,\ldots,a_n,b')$ is feasible. It is either strictly feasible,
in which case $b'\not\in K$, or ill-posed, in which case
$b'\in\partial K$ (use~(\ref{eq:Ksif})). Since $b\in K$ we conclude
that $ d(b,\partial K) \le d(b,b') = \pi/2-\theta$. This implies
$$
\dnew=\sin d(b,\partial K) \le \cos\theta =\<b,q\>
$$
and hence the claimed inequality~(\ref{eq:dpK}). Moreover note that
$d(b,\partial K)\le \pi/2$ and $\dnew>0$ as $b\not\in\partial K$.

Suppose now that $B(p,\r)$ is the SIC for $A$. Since we assume $A$
to be strictly feasible $t:=\cos\r$ is positive. By the
characterization of the GCC condition number in
Theorem~\ref{thm:GCC-dist} we have $t=\sin d(A,\Sigma_{n,m})
=\CC(A)^{-1}$.

Put $\varphi :=\arcsin (\frac1{10} t\dnew)$. For proving the
proposition, it is enough to show the implication
\begin{equation}\label{eq:impl-prop}
  \forall (A',b')\in(S^m)^{n+1} \quad d((A',b'),(A,b)) \le \varphi \Longrightarrow (A',b') \mbox{ infeasible}.
\end{equation}
Indeed, this implies (using $d((A,b),\Sigma_{n+1,m})\le \pi/2$,
cf.~Theorem~\ref{thm:GCC-dist})
$$
 \CC(A,b)^{-1} = \sin d((A,b),\Sigma_{n+1,m}) \ge \sin\varphi = \frac1{10}\, t\,\dnew = \frac1{10}\, \CC(A)^{-1} \dnew,
$$
as claimed in the proposition.

We argue by contradiction. Suppose there is a feasible $(A',b')$
having distance at most $\varphi$ from $(A,b)$. Then there exists
$x'\in S^m$ such that
$$
 \<a'_1,x'\>\le 0,\ldots,\<a'_n,x'\>\le 0,\<b',x'\>\le 0 .
$$
Taking into account that $d(a'_i,a_i)\le\varphi$, we see that
$d(a_i,x')\ge \pi/2-\varphi$ and hence $\<a_i,x'\> \le \sin
\varphi$.

We put now $\tilde{x}:= x' - \lambda p$ with
$\lambda:=t^{-1}\sin\varphi$. As $\<a_i,p\>\ge t$, we have for
$i\in[n]$
$$
 \<a_i,\tilde{x}\> = \<a_i,x'\> - \lambda\<a_i,p\> \le \sin\varphi - \lambda t =0 .
$$
Note that $\tilde{x}\ne 0$ (otherwise $t=\sin\ph$, which is
impossible). Therefore, $\tilde{x}/\|\tilde{x}\|$ is well-defined
and lies in $C$. Inequality~(\ref{eq:dpK}) implies that
$\dnew\|\tilde{x}\|\le \<b,\tilde{x}\>$.

Put  $\Delta b:=b'-b$. Then $\|\Delta b\| \le 2 \sin(\varphi/2)$ by
our assumption $d(b',b)\le\varphi$. We obtain
\begin{eqnarray*}
 \<b,\tilde{x}\> &=& \<b'-\Delta b, x' -\lambda p\>
      = \<b',x'\> - \<\Delta b,x' \> - \<b',\lambda p\> + \<\Delta b,\lambda p\> \\
   & \le & 0 + \|\Delta b\| + \lambda + \|\Delta b\| \lambda .
\end{eqnarray*}
To arrive at a contradiction it is enough to verify that
$$
 \|\Delta b\| + \lambda + \|\Delta b\| \lambda < \dnew \|\tilde{x}\| .
$$
Note that $\|\tilde{x}\|\ge 1-\lambda$, $\|\Delta b\| \le 2$, and
$\dnew\le 1$. It is therefore sufficient to check that
$$
 \|\Delta b\| + \lambda + 2\lambda < \dnew - \lambda,
$$
that is,
$$
 \|\Delta b\| + 4\lambda  < \dnew .
$$
Using $\sin\varphi= 2\sin(\varphi/2)\cos(\varphi/2)$ we get $\lambda
=t^{-1}\sin\varphi \le 2 t^{-1}\sin(\varphi/2)$. It is therefore
sufficient to show that
$$
 2\sin\frac{\varphi}{2} + 8 t^{-1} \sin\frac{\varphi}{2} < \dnew,
$$
which is equivalent to
$$
 (t+4)\sin\frac{\varphi}{2} < \frac12 t\dnew.
$$
As $t\le 1$, it is enough to show that $5\sin\frac{\varphi}{2} <
\frac12 t\dnew$. This is true, since by our assumption
$\sin\frac{\varphi}{2} < \sin\varphi = \frac1{10}\, t\dnew$.
\end{proof}

\subsection{Feasible instances}\label{se:feasible}

We provide here the proof of the part of Theorem~\ref{th:main}
dealing with feasible instances. That is, we wish to show the
claimed bound~(\ref{eq:F}).

Let $\bar{A}\in (S^m)^n$, $0<\a \le \pi/2$, $\s=\sin\a$, and $t\ge
13m(m+1)(2\s\d_c)^{-1}$. Put $\e:=(m+1) t^{-1}$ and $\ph:=\arcsin\e$. We
suppose that $A\in(S^m)^n$ is chosen at random according
to a distribution $\mu_{\bar{A}}$ on $(S^m)^n$
as defined in Section~\ref{subsec:advers-distr}.
Using Proposition~\ref{pro:AF} and the notation introduced there, we
have
$$
 \Prob\{ A\in \Fe_{n,m}^\circ\mbox{ and } \CC(A) \ge t\} \le \sum_{i=1}^n
   \Prob\{A\in \Fe_{n,m}^\circ \mbox{ and } a_i\in  T_o(\partial K_i,\ph)\} .
$$
We first bound the probability on the right-hand side for $i=n$ by
expressing it as an integral over $A':=(a_1,\ldots,a_{n-1})$ of
probabilities conditioned on~$A'$.
Note that $\mu_{\bar{A}}=\mu_{\bar{A}'}\times\mu_{\aa_n}$ where
$\bar{A'}:=(\bar{a}_1,\ldots,\bar{a}_{n-1})$.
Moroever, $A\in\Fe_{n,m}^\circ$ iff $A'\in\Fe_{n-1,m}^\circ$ and $a_n\not\in K_n$,
where we set now $K_{A'}:=K_n=-\scone(A')$, see~(\ref{eq:Ksif}).
This implies
\begin{equation}\label{eq:intf}
\begin{split}
   \Prob_{\mu_{\bar{A}}} &\{A\in  \Fe_{n,m}^\circ \mbox{ and } a_n\in  T_o(\partial K_n,\ph)\}  \\
    =& \Prob_{\mu_{\bar{A}}}\{A'\in \Fe_{n-1,m}^\circ \mbox{ and } a_n\in  T_o(\partial K_{A'},\ph)\}  \\
    =& \underset{A'\in\Fe_{n-1,m}^\circ}{\int}
    \Prob_{\mu_{\aa_n}}\{a_n\in T_o(\partial K_{A'},\ph)\}\, d\mu_{\bar{A}'}.
\end{split}
\end{equation}
We fix $A'\in\Fe_{n-1,m}^\circ$ and consider the properly convex set $K_n=K_{A'}$ in~$S^m$.
The bound in Corollary~\ref{cor:volbdconv} on the outer
neighborhood of $\partial K_n$ yields
$$
  \Prob_\nu\{a_n\in T_o(\partial K_n,\ph)\}  =
  \frac{\vol (T_o(\partial K_n,\ph) \cap B(\bar{a}_n,\a))}{\vol B(\bar{a}_n,\a)} \le
  \frac{13m}{2\s}\, \sin\ph ,
$$
where $\nu$ denotes the uniform distribution on $B(\aa_n,\a)$.
The reader should note that $\e\le 2\s\d_c/(13m)\le\s/(2m)$ by assumption. (We win a
factor of two by considering $B(\bar{a}_n,\a)$ instead of
$B(\pm\bar{a}_n,\a)$.) Hence, using $\sin\ph = \e = (m+1) t^{-1}$,
we conclude
$$
 \Prob_\nu\{a_n\in T_o(\partial K_n,\ph) \} \le
   \frac{13m}{2\s}\, \sin\ph =  \frac{13m(m+1)}{2\s t}.
$$
We assume that $13m(m+1)(2\s t)^{-1}\le\d_c$. 
Hence we can apply Lemma~\ref{lem:CHL:09}, which yields
$$
 \Prob_{\mu_{\aa_n}}\{a_n\in T_o(\partial K_n,\ph) \}
  \le \left(\frac{13m(m+1)}{2\s t}\right)^c.
$$
This bound holds for any $A'\in\Fe_{n-1,m}^\circ$. 
We therefore obtain from (\ref{eq:intf})
\begin{eqnarray}\notag
\lefteqn{\Prob_{\mu_{\bar{A}}}\{ A\in \Fe_{n,m}^\circ \mbox{ and } a_n\in T_o(\partial K_n,\ph)\}} \\
  &\le& \left(\frac{13m(m+1)}{2\s t}\right)^c \Prob_{\mu_{\bar{A}'}}\{A'\mbox{ feasible} \} 
  \ \le\ \left(\frac{13m(m+1)}{2\s t}\right)^c . \label{eq:anschluss}
\end{eqnarray}
The same upper bound holds for any $K_i$. Altogether, we obtain
$$
  \Prob_{\mu_{\bar{A}}}\{ A\in \Fe_{n,m}^\circ\mbox{ and } \CC(A) \ge t\}
  \le n\left(\frac{13m(m+1)}{2\s}\right)^c\, t^{-c} ,
$$
which proves Claim~(\ref{eq:F}), since
$\Prob\{A\in\Sigma_{n,m}\}=0$.

\subsection{Infeasible instances}\label{se:infeasible}

We start with a general observation. For $A=(a_1,\ldots,a_n)\in
(S^m)^n$ and $1\le k\le n$ we will write $A_k:=(a_1,\ldots,a_k)$ and
$\bar{A}_k:=(\bar{a}_1,\ldots,\bar{a}_k)$.

\begin{lemma}\label{le:CCine}
Let $A \in (S^m)^n$, $k <n$, such that $A_{k+1}$ be infeasible. Then
$$
 \CC(A_{k+1}) \ge \CC(A).
$$
\end{lemma}

\begin{proof}
As $A_{k+1}$ is infeasible, $A$ must be infeasible as well. Let
$A'=(a'_1,\ldots,a'_n)$ be feasible such that
$d(A,A')=d(A,\Sigma_{n,m})\le\pi/2$. Then
$A'_k=(a'_1,\ldots,a'_{k+1})$ is feasible and $d(A_k,A'_k)\le
d(A,A')$. Hence we have $d(A_{k+1},\Sigma_{k+1,m})\le
d(A,\Sigma_{n,m})$ and
$$
\CC(A_{k+1})^{-1} = \sin d(A_{k+1},\Sigma_{k+1,m})
   \le \sin d(A,\Sigma_{n,m}) =\CC(A)^{-1},
$$
which was to be shown.
\end{proof}

We also need the following probabilistic lemma.

\begin{lemma}\label{le:multrva}
Let $U$ and $V$ be random variables taking positive values
and $x_U\ge \a>0$, $x_V\ge\b >0$, and $c>0$.
We assume
\begin{eqnarray*}
\Prob\{U\ge x \} &\le& \a\cdot x^{-c} \quad  \mbox{for $x\ge x_U$} \\
\Prob\{V\ge x \mid U\} &\le& \b\cdot x^{-c} \quad  \mbox{for $x\ge x_V$}.
\end{eqnarray*}
Then we have
$$
\Prob\{UV\ge x \} \le c\, \a\b\, x^{-c}\ \ln\max\big\{\frac{x}{x_U x_V},1\big\}
  + \min\{\a\, x_V^c,\b\, x_U^c\}\, x^{-c} .
$$
\end{lemma}

\begin{proof}
\cite[Lemma C.1]{sast:06}
with the functions $f,g$ defined as
$$
f(x) = \left\{\begin{array}{ll} 1 &\mbox{ if $x<x_U$}\\ \a\cdot x^{-c} &
                 \mbox{ if $x\ge x_U$} \end{array}\right.
\quad
g(x) = \left\{\begin{array}{ll} 1 &\mbox{ if $x<x_V$}\\ \b\cdot x^{-c} &
                 \mbox{ if $x\ge x_V$} \end{array}\right.
$$
yields
\begin{equation}\label{eq:UVe}
 \Prob\{UV\ge x \} \le \int_0^\infty f\big(\frac{x}{s}\big) (-g'(s))\, ds .
\end{equation}
If $x\ge x_Ux_V$ we estimate this by
\begin{eqnarray*}
 \Prob\{UV\ge x \} &\le& \int_{x_V}^{x/x_U} \a \, x^{-c}\, s^c\, c\,\b\, s^{-c-1}\, ds +
        \int_{x/x_U}^\infty c\,\b\, s^{-c-1}\, ds  \\
 & = & c\,\a\b\,x^{-c}\, \ln\big(\frac{x}{x_Ux_V}\big) + \b\,x_U^c\,x^{-c} .
\end{eqnarray*}
If $x < x_Ux_V$ one argues similarly.

Finally note that (\ref{eq:UVe}) implies
$\Prob\{UV\ge x \} \le \int_0^\infty g\big(\frac{x}{s}\big) (-f'(s))\, ds$,
using integration by parts.
Estimating this as before, with the roles of $f$ and $g$ exchanged,
completes the proof.
\end{proof}

We provide now the proof of the part of Theorem~\ref{th:main}
dealing with infeasible instances, i.e., of the claimed
bound~(\ref{eq:I}). Fix $\bar{A}\in S^m$, $0<\a\le\pi/2$,
$\s=\sin\a$, and $t\ge 1$. Assume $A=(a_1,\ldots,a_n)$ to be chosen
at random according to $\mu_{\bar{A}}$. Then $A_{m+1}$ is always
feasible. Hence, if $A=A_n$ is infeasible, then there exists a
smallest index~$k > m$ such that $A_k$ is feasible and $A_{k+1}$ is
infeasible. If we denote by $\EE_k$ the event
\begin{equation}\label{eq:event}
 \mbox{$A_k$ feasible and $A_{k+1}$ infeasible and $\CC(A_{k+1})\ge t$},
\end{equation}
and take into account Lemma~\ref{le:CCine}, we obtain
\begin{equation}\label{eq:EEstep}
 \Prob\{ A\in \In_{n,m}\mbox{ and } \CC(A) \ge t\} \le
 \sum_{k=m+1}^{n-1}\Prob\EE_k .
\end{equation}

For bounding the probability of $\EE_k$, a change of notation is
convenient. We fix $k$ and write from now on
$$
 A:=(a_1,\ldots,a_k),\quad  K_A:=-\scone\{a_1,\ldots,a_k\},\quad  b:=a_{k+1},
$$
and similarly $\bar{A}:=(\bar{a}_1,\ldots,\bar{a}_k)$,
$\bar{b}:=\bar{a}_{k+1}$. We note that $A$ and $b$ are chosen
independently and at random according to $\mu_{\bar{A}}$ and
$\mu_{\bar{b}}$, respectively.

Proposition~\ref{pro:IF} implies that
$$
 \Prob_{\mu_{\bar{A}}}\EE_k \le \Prob_{\mu_{\bar{A}}}\left\{A\in\Fe_{k,m} \mbox{ and } b\in K_A \mbox{ and }
    \frac{\CC(A)}{\sin d(b,K_A)} \ge \frac{t}{10} \right\} .
$$
The first part of Theorem~\ref{th:main} tells us that
\begin{equation}\label{eq:infeas-red-feas}
 \Prob_{\mu_{\bar{A}}}\big\{A\in\Fe_{k,m} \mbox{ and } \CC(A)\ge x\big\} \le
 k\left(\frac{13m(m+1)}{2\s}\right)^c\, x^{-c}
\end{equation}
provided $x\ge x_U:=13m(m+1)/(2\s\d_c)$. For a fixed strictly feasible $A$,
the set $K_A$ is properly convex in $S^m$. The bound in
Corollary~\ref{cor:volbdconv} on the inner neighborhood of $\partial
K_A$ yields for any $A\in\Fe_{k,m}^\circ$
\begin{equation}\label{eq:star}
 \underset{\nu}{\Prob}\Big\{b\in K_A \mbox{ and } \frac1{\sin d(b,\partial K_A)} \ge x\ \Big|\ A \Big\}
 \le \frac{13m}{2\s} \frac1x ,
\end{equation}
provided $x\ge 2m/\s$ (again $\nu$ denotes the uniform distribution on $B(\bar{b},\a)$).
Applying Lemma~\ref{lem:CHL:09} yields
\begin{equation}\label{eq:star-mu}
 \underset{\mu_{\bar{b}}}{\Prob}\Big\{b\in K_A \mbox{ and } \frac1{\sin d(b,\partial K_A)} \ge x\ \Big|\ A \Big\}
 \le \left(\frac{13m}{2\s}\right)^c x^{-c} ,
\end{equation}
provided $x\ge x_V:= 13m/(2\s\d_c)$.

Let $\bI_M$ denote the indicator function of a set $M$. We combine
the above two probability estimates in~\eqref{eq:infeas-red-feas}
and~\eqref{eq:star-mu} with Lemma~\ref{le:multrva}, setting
$$
 U(A):=\bI_{\Fe_{k,m}}(A)\, \CC(A),\quad
 V(A,b):= \bI_{K_A}(b)\, \frac1{\sin d(b,\partial K_A)} .
$$
Note that
$$
 \Prob\EE_k \le \Prob\{ U(A)\cdot V(A,b) \ge t/10 \}.
$$
We have for $t\ge 1$ and $x=t/10$
$$
 \max\left\{\frac{x}{x_U x_V},1\right\} =
    \max \left\{ \frac{4\,\s^2\,\d_c^2\, t}{1690\,m^2(m+1)}, 1\right\} \le t .
$$
Let $\alpha:=k\left(\frac{13 m(m+1)}{2\s}\right)^c$ and
$\beta:=\left(\frac{13m}{2\s}\right)^c$. Note that
$\a\b=k\cdot\left(\frac{169\,m^2(m+1)}{4\s^2}\right)^c$
and $\min\{\alpha x_V^c,\beta x_U^c\}=\left(\frac{169 m^2(m+1)}{4\s^2\d_c}\right)^c$.
Lemma~\ref{le:multrva} implies that
\begin{equation}\label{eq:interface}
 \Prob\EE_k \le c\, k\, \left(\frac{1690\,m^2(m+1)}{4\,\s^2}\right)^c\,t^{-c}\, \ln t
  + \left(\frac{1690\,m^2(m+1)}{4\s^2\,\d_c}\right)^c\, t^{-c} .
\end{equation}
Plugging in this bound into~(\ref{eq:EEstep}) finishes the proof of
Theorem~\ref{th:main}. \hfill$\Box$


\subsection{Proof of Theorem~\ref{th:Emain}}

We will now give the proof of Theorem~\ref{th:Emain} by
deriving estimates of the expectation of $\ln\CC$ from the tail bounds
given in Theorem~\ref{th:main}.
Let $A\in(S^m)^n$ be chosen at random according
to an adversarial distribution $\mu_{\bar{A}}$.
Combining~\eqref{eq:F} and~\eqref{eq:I} from Theorem~\ref{th:main}
we have for $t\ge 13m(m+1)(2\s\d_c)^{-1}$,
\begin{eqnarray*}
   \Prob\{ \CC(A) \ge t\} & = &
       \Prob\{ A\in \Fe_{n,m} \mbox{ and } \CC(A) \ge t\}
\\ & & \qquad\qquad + \Prob\{ A\in \In_{n,m} \mbox{ and } \CC(A) \ge t\}
\\ & \le & \textstyle n\, \left(\frac{13m(m+1)}{2\s}\right)^c \, t^{-c}
           + n\, \left(\frac{1690m^2(m+1)}{4\s^2\d_c}\right)^c \, t^{-c}
\\ & & \qquad \textstyle + c n^2\, \left(\frac{1690m^2(m+1)}{4\s^2}\right)^c \, t^{-c}\, \ln t .
\end{eqnarray*}
Defining $D_1:=n \left(\frac{13m(m+1)}{2\s}\right)^c$,
$D_2:=D_1\cdot\left(\frac{130 m}{2\s\d_c}\right)^c$,
$D_3:=c n\,D_1\cdot\left(\frac{130 m}{2\s}\right)^c$
we can write shortly
\begin{eqnarray*}
   \Prob\{ \CC(A) \ge t\} & \le & D_1\, t^{-c}
                 + D_2\, t^{-c} + D_3\, t^{-c} \ln t .
\end{eqnarray*}
Note that $D_2\geq D_1$ and $D_3\geq D_1$.
It is convenient to use tail estimates of $\ln\CC$ instead of $\CC$, so we reformulate
\begin{equation}\label{eq:tail-estim-mu-exp}
   \Prob\{ \ln\CC(A) \ge s\} \le D_1\, e^{-c s}
                 + D_2\, e^{-c s} + D_3\, e^{-c s} s ,
\end{equation}
which holds for $s\geq \ln\left(\frac{13 m(m+1)}{2\s\d_c}\right)$.
We define
$s_i := \ln D_i^{1/c} + \ln \d_c^{-1}$,
i.e.~$\d_c = D_i^{1/c}\cdot e^{-s_i}$, for $i=1,2,3$.
Note that $s_2 = s_1 + \ln\left(\frac{130 m}{2\s\d_c}\right)\geq s_1$, $s_3= s_1 
 + c^{-1}\,\ln(c n) + \ln\left(\frac{130 m}{2\s}\right) \geq s_1$, and
$$
 s_1 = \textstyle\frac{1}{c} \ln n + \ln \left(\frac{13 m(m+1)}{2\s}\right) + \ln \d_c^{-1}
     \ge \ln \left(\frac{13 m(m+1)}{2\s\d_c}\right) .
$$
The definition of $\d_c$ in Lemma~\ref{lem:CHL:09} 
and the inequality
\begin{equation}\label{eq:est-CHL}
  \textstyle \left(1-\left(\frac{2}{\pi m}\right)^{\frac{1}{m}}\right)^{-1/2}
\le\sqrt{\frac{2m}{\ln(\pi m/2)}} ,
\end{equation}
which is shown by a small computation,
lead to the following estimate of $s_1$
\begin{equation}\label{eq:estim-s_e}
\begin{array}{rcl}
  s_1 & \le & \ln D_1^{1/c} + \ln \frac{\pi m}{2} + \frac{1}{2c}\ln\left(\frac{2H^2 m}{\ln(\pi m/2)}\right)
\\[2mm] & = & c^{-1} \ln n + \ln\left(\frac{13 m(m+1)}{2\s}\right) 
 + \ln \frac{\pi m}{2} + \frac{1}{2c}\ln\left(\frac{2H^2 m}{\ln(\pi m/2)}\right) \; .
\end{array}
\end{equation}
We distinguish the cases $s_3\geq s_2$ and $s_3<s_2$. For $s_3\geq s_2$ we get with~\eqref{eq:tail-estim-mu-exp}
\begin{eqnarray*}
   \bE\big(\ln\CC\big) & = & \int_0^\infty
                 \Prob\{\ln\CC(A) \ge s\} \, ds
\\ & \le & \int_0^{s_3} 1\, ds +
  \int_{s_3}^\infty \big(D_1\, e^{-c s} + D_2\, e^{-c s} + D_3\, e^{-c s} \, s \big) \,ds
\\ & \le & s_3 + \int_{s_1}^\infty D_1\, e^{-c s} \,ds + \int_{s_2}^\infty D_2\, e^{-c s} \,ds
                + \int_{s_3}^\infty D_3\, e^{-c s} \, s \,ds
\\ & = & s_3 +  c^{-1}\d_c^c +c^{-1}\d_c^c  + c^{-2}(1+cs_3) \d_c^c
\\ & \leq & (1+c^{-1})\, s_3 + 2c^{-1} + c^{-2}
\\ & = & \textstyle(1+c^{-1})\, s_1
              + (1+c^{-1})c^{-1}\,\ln(c n)+ (1+c^{-1})\ln\left(\frac{130 m}{2\s}\right)
\\ & & \qquad + 2c^{-1} + c^{-2} .
\end{eqnarray*}
With the estimate of $s_1$ in \eqref{eq:estim-s_e} a small computation yields
\begin{eqnarray*}
   \bE\big(\ln\CC\big) & \le & (1+c^{-1})\Big(2 c^{-1}\ln n + 3\ln m + \ln (m+1) + 2\ln \frac{1}{\s}
\\ & & \qquad + (2c)^{-1}\ln\left(\frac{2 H^2 m}{\ln(\pi m/2)}\right) + c^{-1} + c^{-1}\ln c + 6.5 \Big) + c^{-1} .
\end{eqnarray*}
This yields
$\bE\big(\ln\CC\big) =\Oh\big( c^{-2} \ln\big(\frac{nH}{\s}\big)\big)$.
In the case $s_3 <s_2$ a similar argument holds.
This shows the first statement of Theorem~\ref{th:Emain}.
Finally, tracing the constants in the case $\b=0$,
yields the asserted explicit bound.
\hfill$\Box$

\subsection{Average-case analysis}\label{se:average}

We show here that for the uniform distribution on $S^m$,
the probability tail estimates in Theorem~\ref{th:main} on $\CC(A)$
can be significantly improved by essentially the same method.

\begin{proposition}\label{pro:av}
For $A\in (S^m)^n$ chosen uniformly at random we have 
$$
 \Prob\{\CC(A)\ge t \}\ =\ \Oh\big((m+1)^5\, \frac1t\, \ln t\big) .
$$
Moreover, $\bE(\ln\CC(A)) =\Oh(\ln m)$ as stated in Corollary~\ref{cor:average}.
\end{proposition}

\begin{proofsketch}
A result due to Wendel~\cite{wend:62} states that for $k>m$
\begin{equation}\label{eq:wendel}
 p(k,m) := \frac{\vol\Fe_{k,m}}{(\vol S^m)^k} = \frac1{2^{k-1}}\sum_{i=0}^m{k-1\choose i} .
\end{equation}
We refer to the proof in~\S\ref{se:feasible} for the uniform distribution on $S^m$
instead of $\mu_{\bar{A}}$ (think of $\s=1$). We write $k$ instead of $n$.
We do not need to use Lemma~\ref{lem:CHL:09}. Moreover we do not bound
$p(k,m)=\Prob\{A\in\Fe_{k,m}\}$ by~$1$ as in Equation~\eqref{eq:anschluss}.
Taking account of this, the proof in~\S\ref{se:feasible} shows that
$$
 \Prob\big\{A\in\Fe_{k,m} \mbox{ and } \CC(A)\ge x \big\} \le
  k\cdot \frac{13m(m+1)}{2}\,p(k,m)\ \frac1{x},
$$
provided $x\ge x_U:=2m(m+1)$.

We proceed now as in \S\ref{se:infeasible}, using the same notation.
For a fixed strictly feasible $A=(a_1,\ldots,a_{k})$ we have by
(\ref{eq:star})
$$
 \Prob\Big\{ b\in K_A \mbox{ and } \frac1{\sin d(b,\partial K_A)} \ge x\ \Big|\  A\Big\}
 \le \frac{13m}{2} \frac1x
$$
provided $x\ge x_V:=2m$.
Recall the definition of the event $\EE_k$
from (\ref{eq:event}).
Similarly as for (\ref{eq:interface}) we
conclude with the help of Lemma~\ref{le:multrva} that
$$
 \Prob\EE_k \le  \cnew m^3\, k\, p(k,m)\ \frac1t\, \ln t \quad \mbox{ for $t\ge e$,}
$$
where $\cnew$ stands for a universal constant.
Using Lemma~\ref{le:pkm} stated below we get
$$
\sum_{k=m+1}^{n-1} k\, p(k,m)  \le \sum_{k=m+1}^{4m} k\, p(k,m) +
 \sum_{k=4m+1}^{\infty} k\, p(k,m) =  \Oh(m^2) .
$$
Hence, by Equation~(\ref{eq:EEstep}),
\begin{equation*}
 \Prob\{ A\in \In_{n,m}\mbox{ and } \CC(A) \ge t\} \le
 \sum_{k=m+1}^{n-1}\Prob\EE_k \le
 \cnew' m^5\,  \frac1t\, \ln t
 \end{equation*}
for some constant $\cnew'$. It is obvious that $\Prob\{ A\in
\Fe_{n,m}\mbox{ and } \CC(A) \ge t\}$ can also be bounded this way.
Finally, the claimed bound on the expectation of $\ln\CC(A)$
follows immediately with the help of \cite[Prop.~2.4]{BCL:06a}.
\end{proofsketch}

\begin{lemma}\label{le:pkm}
We have $\sum_{k=4m+1}^\infty k\,p(k,m) = o(1)$ for $m\to\infty$.
\end{lemma}

\begin{proof}
Let $k>4m$. Wendel's result~(\ref{eq:wendel}) implies
$$
 k\, p(k,m) \le  k \frac{(m+1)}{2^{k-1}}\, {k-1\choose m}
 \le \frac{2(m+1)}{m!}\, \frac{k^{m+1}}{2^k} .
$$
We have $k^{m+1}2^{-k} \le 2^{-k/2}$ for $k\ge \cnew m\log m$, and
sufficiently large $m$, where $\cnew>0$ is a suitable universal
constant. Therefore, we get
$$
 \sum_{k\ge \cnew m\log m} k\, p(k,m)\ \le\
  \frac{2(m+1)}{m!}\sum_{k=0}^\infty \frac1{2^{k/2}}\  =\ o(1)\ (m\to\infty) .
$$
The function $x\mapsto x^{m+1}2^{-x}$ is monotonically decreasing
for $x\ge (m+1)/\ln2$. Hence, as $k> 4m$,  and using $m!\ge (m/e)^m$
we get
$$
 \frac1{m!}\,\frac{k^{m+1}}{2^k}\ \le\
 \frac1{m!}\,\frac{(4m)^{m+1}}{2^{4m}} \ \le\
 4m\, \big(\frac{e}{4}\big)^m .
$$
Since $e/4<1$, we conclude
$$
 \sum_{ k= 4m+1}^{\cnew m\log m} k\, p(k,m)\,
 \le \  8m(m+1) \big(\frac{e}{4}\big)^m \cnew m\log m\
 \ =\ o(1)\ (m\to\infty),
$$
which completes the proof.
\end{proof}

\section{Some spherical convex geometry}\label{se:volnbS}

The goal of this section is to provide the proof of Theorem~\ref{th:volbdconv},
following the lines of~\cite[Theorem~1.2]{BCL:08}.
We proceed in several steps.

\subsection{Integrals of curvature and Weyl's tube formula}
\label{se:weyl}

For the following material from differential geometry we refer to~\cite{spiv3:79}.
A good reference for the differential geometry of convex sets is~\cite{bofe:74}.

Let $V$ be a smooth hypersurface in $S^m$ with unit normal vector field
$\nu\colon V\to S^m$. The {\em principal curvatures} of $V$ at~$x\in V$
are defined as the eigenvalues $\kappa_1(x),\ldots,\kappa_{m-1}(x)$ of
the Weingarten map
$-D\nu(x)\colon T_xV\rightarrow T_xV$.
The  $i$th curvature $\scK_{V,i}(x)$ of $V$ at $x$ is the $i$th symmetric polynomial
in the principal curvatures:
$$
 \scK_{V,i}(x) :=\sum_{|I|=i} \prod_{j\in I}\kappa_j(x) \quad (0\le i<m).
$$
Interesting special cases are $\scK_{V,0}(x)=1$ and
\begin{equation}\label{eq:gk}
\scK_{V,m-1}(x)=\kappa_1(x)\cdots\kappa_{m-1}(x) =\det (-D\nu(x)),
\end{equation}
which is called the {\em Gaussian curvature} of $V$ at~$x$.
Let $U$ be an open subset of~$V$.
In~\cite{BCL:08} the integral $\mu_i(U)$ of $i$th curvature and the
integral $|\mu_i|(U)$ of $i$th absolute curvature were defined as
$$
 \mu_i(U)  := \int_U \scK_{V,i}\, dV, \quad
 |\mu_i|(U) := \int_U |\scK_{V,i}|\, dV  .
$$
Two special cases deserve special mention:
$\mu_0(U)=\vol\, U$ equals the $(m-1)$-dimensional volume of $U$.
Moreover,
$\mu_{m-1}(V)$ is the
{\em integral of the Gaussian curvature} of $V$.

By a {\em smooth convex body} $K$ in $S^m$
we will understand a convex body
such that its boundary $\partial K$ is a smooth hypersurface
in $S^m$ (of type $C^\infty$) and its Gaussian curvature
does not vanish in any point of $\partial K$.

Let $K$ be a smooth convex body in $S^m$ with boundary $V:=\partial K$.
We denote by $\nu\colon V\to S^m$ the unit normal vector field of
the hypersurface~$V$ that points inwards of $K$.
Here all the principal curvatures $\kappa_j(x)$ are nonnegative, cf.~\cite{bofe:74}.
Hence the $i$th curvatures are nonnegative as well and therefore
we have $\mu_i(U) =|\mu_i|(U)$ for any open subset $U$ of $V$.

For $0<\ph\le\pi/2$ we define the
{\em $\ph$-tube} $T^\perp(U,\ph)$ around~$U$ as
\begin{align*}
 \Tusn(U,\ph) := \{x\in S^m \mid\;& \mbox{$\exists y\in U$ such that $d(x,y)<\ph$ and} \\
  & \mbox{$[x,y]$ intersects $U$ orthogonally at $y$}\}.
\end{align*}
The {\em outer $\ph$-tube} $ \Tusn_o(U,\ph)$
and {\em inner $\ph$-tube}  $\Tusn_i(U,\ph)$ of $U$
are defined as
$$
 \Tusn_o(U,\ph) := \Tusn(U,\ph)\setminus K
 \mbox{ and }
 \Tusn_i(U,\ph) := \Tusn(U,\ph)\cap K .
$$

In an important paper, Weyl~\cite{weyl:39} derived a formula
for the volume of tubes around compact submanifolds of
euclidean spaces or spheres.
His result can be seen as extension of Steiner's formula
on the volume of ``parallel convex sets'' in euclidean space,
see also Allendoerfer~\cite{alle:48}.
Weyl's formula only holds for a sufficiently small radius.
In \cite[Prop.~3.1]{BCL:08}, it was observed that
when replacing integrals of curvature by absolute integrals
of curvature, one obtains an upper bound on
the volume of tubes holding for any radius.
As the above two notions of curvature coincide for boundaries
of convex sets, we get the following result.
(An inspection of the proof of \cite[Prop.~3.1]{BCL:08}
reveals that separate bounds on the inner and outer tube hold.)

\begin{proposition}\label{pro:tube-vol-mc}
Let $K$ be a smooth convex body in $S^m$
and $U$ be an open subset of $\partial K$.
Then we have for all $0<\ph \le \pi/2$
\begin{equation*}
 \max\{\vol \Tusn_o(U,\ph), \vol \Tusn_i(U,\ph)\}
 \ \le\  \sum_{i=0}^{m-1} J_{m,i+1}(\ph)\, \mu_i(U) ,
\end{equation*}
where $J_{m,i+1}$ denotes the function defined in~\eqref{eq:def-J-fcts}.
Moreover, this upper bound is sharp for sufficiently small $\ph$,
cf.~\cite{weyl:39}.
\end{proposition}

\subsection{Some integral geometry}

We will need a special case of the principal kinematic
formula of integral geometry for spheres.
We denote by $G$ the orthogonal group $O(m+1)$,
that operates on $S^m$ in the natural way,
and denote by $dG$ its volume element normalized
such that the volume of $G$ equals one.
The following result is Theorem~2.7. in~\cite{BCL:08}.
(For related information see \cite{howa:93} and \cite{glas:95}.)

\begin{theorem}\label{th:kin-form-S}
Let $U$ be an open subset of a compact oriented smooth hypersurface~$M$
of $S^m$ and $0\le i <m-1$. Then we have
\begin{equation*}
\mu_i(U) = \kc(m,i) \int_{g\in G} \mu_i(gU\cap S^{i+1})\, dG(g),
\end{equation*}
where
$\kc(m,i) = (m-i-1){m-1\choose i}\,\frac{\Oh_{m-1}\Oh_{m}}
  {\Oh_{i}\Oh_{i+1}\Oh_{m-i-2}}$.
\end{theorem}

The special case $i=0$ yields an effective tool for estimating volumes,
usually referred to as Poincar\'e's formula:
\begin{equation}\label{eq:crofton}
 \vol_{m-1} U
 = \frac{\Oh_{m-1}}{2} \int_{g\in G} \#(U \cap gS^{1})\, dG(g),
\end{equation}
where $\#(U \cap gS^{1})$ denotes the number of elements
in $U \cap gS^{1}$ (note that this is a finite set
for almost all $g\in G$).
Here is an application of (\ref{eq:crofton}).
Clearly, the given bound is sharp (consider
spherical caps with radius almost $\pi/2$).

\begin{corollary}\label{cor:volconv}
Any smooth convex body $K$ in $S^m$ satisfies
$\vol\,\partial K\le \Oh_{m-1}$.
\end{corollary}

\begin{proof}
Almost surely, the intersection $\partial K \cap gS^1$ is finite.
Then it consists of at most two points by convexity.
\end{proof}

\subsection{Integrals of curvature for boundaries of convex sets}

In this section we assume that $K$ is a smooth convex body in $S^m$
and $\nu$ is the unit normal vector field on
$\partial K$ pointing inwards of $K$.
This means that for $x\in V$,
the unit vector $\nu(x)$ is uniquely characterized by the conditions
$\<v,x\>=0$ and $\<v,y\>\ge 0$ for all $y\in K$.

\begin{lemma}\label{le:imnu}
We have $-\nu(\partial K)= \partial\breve{K}$.
\end{lemma}

\begin{proof}
The characterization of $\nu(x)$ implies that
$-\nu(x)\in\partial\breve{K}$ for $x\in \partial K$.
For the other inclusion, let $v$ be a unit vector satisfying $-v\in\partial\breve{K}$.
Then $\<v,y\>\ge 0$ for all $y\in K$. Moreover, there exists $x\in K$ such that $\<v,x\>=0$.
This implies $x\in\partial K$. It follows that $v=\nu(x)$.
\end{proof}

The following bound is crucial for all what follows.
Again, considering spherical caps with radius almost $\pi/2$,
shows the optimality of the bound.

\begin{proposition}\label{le:GCbound}
The integral of Gaussian curvature of $\partial K$ is bounded as
$\mu_{m-1}(\partial K)\le\Oh_{m-1}$.
\end{proposition}

\begin{proof}
Put $V:=\partial K$. By Lemma~\ref{le:imnu},
$\nu\colon V\to \partial\breve{K}$ is surjective.
By (\ref{eq:gk}) we have $\scK_{V,m-1}(x)=\det (-D\nu(x))$ for $x\in V$.
Since we assume that the Gaussian curvature does not vanish,
the map $\nu$ has no singular values.

We claim that $\nu$ is injective. Otherwise, we had $\nu(x)=\nu(y)$ for distinct $x,y\in V$.
Since  $\lan \nu(x),x\ran=0$ and $\lan \nu(y),y\ran=0$ we had
$\lan \nu(x),z\ran=0$ for all $z\in [x,y]$.
Hence $\nu$ would be constant along this segment and therefore $x$ would be a critical point,
contradicting our asumption.

We conclude that $-\nu\colon V\to \nu(V)$ is a diffeomorphism
onto the smooth hypersurface $\partial\breve{K}$.
The transformation theorem yields
$$
 \mu_{m-1}(V) =  \int_V \scK_{V,m-1}\, dV = \int_V \det(-D\nu)\, dV = \vol\,\partial\breve{K}.
$$
Corollary~\ref{cor:volconv} implies now the assertion.
\end{proof}

\begin{lemma}\label{le:mc-estimate}
For $a\in S^m$, $0<\a\le\pi/2$, $\s=\sin\a$,
and $0\le i < m$ we have
\begin{equation*}
 \mu_i(\partial K\cap B(a,\a))\ \le\
  {m-1\choose i} \Oh_{m-1}\, \s^{m-i-1} .
\end{equation*}
\end{lemma}

\begin{proof}
This is similar, but somewhat simpler than the proof of
\cite[Prop.~3.2]{BCL:08}.
The case $i=m-1$ is already established by Proposition~\ref{le:GCbound}.
Hence we assume $i<m-1$.
Let $g\in G=O(m+1)$ be such that $V:=\partial K$ intersects $gS^{i+1}$ transversally
with nonempty intersection. We apply Proposition~\ref{le:GCbound}
to the convex body $K\cap gS^{i+1}$ in the sphere $gS^{i+1}$,
which has the smooth boundary $V\cap gS^{i+1}$.
Hence $\mu_i(V\cap gS^{i+1}) \le \Oh_i$.
The kinematic formula of Theorem~\ref{th:kin-form-S} applied to the
open subset $U:=V\cap \inte(B(a,\a))$ of $V$ yields
\begin{equation*}
\begin{split}
  \mu_i(U) &= \kc(m,i) \int_{g\in G} \mu_i(gU\cap S^{i+1}) \, dG(g) \\
     &\le \kc(m,i)\, \Oh_i\, \Prob_{g\in G}\{gU\cap S^{i+1}\ne\emptyset\}.
\end{split}
\end{equation*}
Using $gU\subseteq B(ga,\a)$, this probability may be estimated
as follows
\begin{eqnarray*}
\lefteqn{\Prob_{g\in G}\{gU\cap S^{i+1}\ne\emptyset\}\ \le\
   \Prob_{g\in G} \{B(ga,\a)\cap S^{i+1}\ne\emptyset\} } \\
 &= &  \Prob_{a'\in S^m} \{B(a',\a)\cap S^{i+1}\ne\emptyset\}
   = \Oh_{m}^{-1}\vol T(S^{i+1},\a) .
\end{eqnarray*}
Lemma~2.1 in~\cite{BCL:08} implies
$\vol\, T(S^{i+1},\a) = \Oh_{i+1}\Oh_{m-i-2}\, J_{m,m-i-1}(\a)$.
Moroever, Lemma~2.2 in~\cite{BCL:08} says that
\begin{equation}\label{eq:JE}
 J_{m,k}(\a) \le \frac{\s^{k}}{k} \quad  \mbox{ for $1\le k<m$.}
\end{equation}
By combining these estimates and plugging in
the formula for $\kc(m,i)$ from Theorem~\ref{th:kin-form-S},
the resulting expression considerably simplifies and we get
$\mu_i(U)\ \le\ {m-1\choose i} \Oh_{m-1}\ \s^{m-i-1}$
as claimed.
\end{proof}

\subsection{Proof of Theorem~\ref{th:volbdconv}}

We can finally provide the proof of Theorem~\ref{th:volbdconv}.
We assume first that $K$ is a smooth convex body in $S^m$.
Let $a\in S^m$, $0<\a,\ph\le\pi/2$, put $\s=\sin\a$, $\e=\sin\ph$,
and let $U=\partial K\cap B(a,\a)$.
By combining Proposition~\ref{pro:tube-vol-mc}
with Lemma~\ref{le:mc-estimate} we get
\begin{equation*}
  \vol T_o^\perp(U,\ph) \le\
   \sum_{i=0}^{m-1} {m-1\choose i} \Oh_{m-1} \sigma^{m-i-1} J_{m,i+1}(\ph).
\end{equation*}
Using the estimate (\ref{eq:JE}) we obtain after a short calculation
(put $k=i+1$, use ${m-1\choose k-1}=\frac{k}{m}{m\choose k}$ and
consider separately the term for $k=m$)
\begin{equation}\label{eq:tubel}
 \vol\Tusn_o(\partial K\cap B(a,\a),\ph) \ \le\
 \frac{\Oh_{m-1}}{m} \sum_{k=1}^{m-1} {m\choose k}\, \e^{k}\,\sigma^{m-k}
 \ +\  \frac12 \Oh_{m}\, \e^m .
\end{equation}
The same upper bound holds for the volume of
$\Tusn_i(\partial K\cap B(a,\a),\ph)$.

We claim that
\begin{equation}\label{eq:gopfh}
T_o(\partial K,\ph)\cap B(\pm a,\a)\subseteq T_o^\perp(\partial K\cap B(\pm a,\b),\ph)
\end{equation}
where $\b=\arcsin\min\{1,\s+\e\}$.
Indeed, suppose $x\in T_o(\partial K,\ph)\cap B(\pm a,\a)$ and let $y\in\partial K$
be a closest point to $x$.
Then $d(x,y)\le\ph$ and $[x,y]$
intersects $\partial K$ orthogonally (as $\partial K$ is smooth without boundary).
The triangle inequality for projective distance (cf.~\S\ref{se:dist})
implies that
$\sin d(a,y) < \sin d(a,x) + \sin d(x,y)\le \s + \e$.
Hence $\sin d(a,y) \le \sin\b$ and therefore
$y\in B(\pm a,\b)$ which shows the claim.

By combining (\ref{eq:gopfh}) with (\ref{eq:tubel}) we get
$$
 \vol (T_o(\partial K,\ph)\cap B(\pm a,\a)) \le
  \frac{2\Oh_{m-1}}{m} \sum_{k=1}^{m-1} {m\choose k}\, \e^{k}\,(\sigma+\e)^{m-k}
     \ +\  \Oh_{m}\, \e^m .
$$
We have $\vol B(\pm a,\a)\ge 2\Oh_{m-1}\frac{\s^m}{m}$,
cf.~\cite[Lemmas~2.1-2.2]{BCL:08}.
Using this, we obtain
\begin{equation}\label{eq:almostdone}
 \frac{\vol (T_o(\partial K,\ph) \cap B(\pm a,\a))}{\vol B(\pm a,\a)} \le
   \sum_{k=1}^{m-1} {m\choose k}\, \Big(1 + \frac{\e}{\s}\Big)^{m-k}\, \Big(\frac{\e}{\s}\Big)^k
    + \frac{m\Oh_m}{2\Oh_{m-1}}\, \Big(\frac{\e}{\s}\Big)^m .
\end{equation}
This shows the assertion of Theorem~\ref{th:volbdconv} for the
outer neighborhood in the case where $K$ is a smooth convex body.
The bound for the inner neighborhood is shown similarly.

The general case where $K$ is any properly convex set in $S^m$
will follow by a pertubation argument.
We define the {\em Hausdorff distance} $d(K,K')$
of two convex sets $K$ and~$K'$ in $S^m$
as the infimum of the real numbers $\d\ge 0$ satisfying
$K\subseteq T(K',\d)$ and $K'\subseteq T(K,\d)$.
This defines a metric and allows to speak about
the convergence of convex sets.
(For compact
convex sets in euclidean space the Hausdorff distance
is a well known notion.)

\begin{lemma}\label{le:pert}
Any properly convex set $K$ in $S^m$ is the limit of a sequence
of smooth convex bodies.
\end{lemma}

\begin{proof}
The euclidean version of the claim is
a well known result due to Minkowski,
see~\cite[\S6]{bofe:74} (or~\cite{Schn:84})
for more information.

A properly convex set $K\subset S^m$ is contained in an open halfspace.
For fixed $p\in S^m$ consider now the open halfsphere
$S^m_+:=\{ x\in S^m\mid  \<x,p\> > 0\}$
with center $p$ and the affine space
$E:=\{x\in\R^{m+1} \mid \<x,p\> = 1\}$.
The ``perspective map''
$\pi\colon S^m_+\to E, x\mapsto \<p,x\>^{-1} x$
maps an intersection of a linear space with $S^m$ to an affine linear subspace of $E$
and vice versa. Moreover, $\pi$
maps convex sets to convex sets
and vice versa.
Moreover, one sees that
$\pi$ induces a homeomorphism between
the set of convex subsets of $S^m_+$ and
the set of compact convex subsets of $E$.
It is easily checked that if $\tilde{K}\subset E$ smooth compact convex
has nonvanishing Gaussian curvature on the boundary, then
this also holds for $\pi(\tilde{K})$.
The assertion follows from the
euclidean version of our claim.
\end{proof}

To finish the proof of Theorem~\ref{th:volbdconv}
let now $K\subset S^{m}$ be a properly convex set and $\d>0$.
By Lemma~\ref{le:pert} there exists a smooth convex body~$K'$
such that $K$ and $K'$ have
Hausdorff distance at most $\d$, which means that
$K\subseteq T(K',\d)$ and $K'\subseteq T(K,\d)$.
This implies $K' \setminus K\subseteq T(\partial K,\d)$ and
$$
 T_o(\partial K,\ph)\subseteq T_o(\partial K',\ph+\d) \cup (K' \setminus K) .
$$
By applying (\ref{eq:almostdone}) to $T_o(\partial K',\ph+\d)$,
letting $\d\to 0$, and noting that
$\vol T(\partial K,\d)$ goes to zero,
the desired assertion follows.
For the inner neighborhood one argues similarly.
\hfill$\Box$

{\small

\begin{thebibliography}{10}

\bibitem{alle:48}
C.B. Allendoerfer.
\newblock Steiner's formulae on a general {$S^{n+1}$}.
\newblock {\em Bull. Amer. Math. Soc.}, 54:128--135, 1948.

\bibitem{ball:93}
K.~Ball.
\newblock The reverse isoperimetric problem for {G}aussian measure.
\newblock {\em Discrete Comput. Geom.}, 10(4):411--420, 1993.

\bibitem{bofe:74}
T.~Bonnesen and W.~Fenchel.
\newblock {\em Theorie der konvexen {K}\"orper}.
\newblock Springer-Verlag, Berlin, 1974.
\newblock Berichtigter Reprint.

\bibitem{buer:09a}
P.~B\"urgisser.
\newblock Smoothed analysis of condition numbers.
\newblock In {\em Foundations of Computational Mathematics, Hong Kong 2008},
  pages 1--41. Cambridge University Press, 2009.

\bibitem{BCL:08a}
P.~B\"urgisser, F.~Cucker, and M.~Lotz.
\newblock Coverage processes on spheres and condition numbers for linear
  programming.
\newblock To appear in {\em Annals of Probability}.

\bibitem{BCL:06a}
P.~B\"urgisser, F.~Cucker, and M.~Lotz.
\newblock Smoothed analysis of complex conic condition numbers.
\newblock {\em J. Math. Pures et Appl.}, 86:293--309, 2006.

\bibitem{BCL:08}
P.~B{\"u}rgisser, F.~Cucker, and M.~Lotz.
\newblock The probability that a slightly perturbed numerical analysis problem
  is difficult.
\newblock {\em Math. Comp.}, 77(263):1559--1583, 2008.

\bibitem{ChC:01}
D.~Cheung and F.~Cucker.
\newblock A new condition number for linear programming.
\newblock {\em Math. Program.}, 91(1, Ser. A):163--174, 2001.

\bibitem{ChC01}
D.~Cheung and F.~Cucker.
\newblock Probabilistic analysis of condition numbers for linear programming.
\newblock {\em Journal of Optimization Theory and Applications}, 114:55--67,
  2002.

\bibitem{ChCH:05}
D.~Cheung, F.~Cucker, and R.~Hauser.
\newblock Tail decay and moment estimates of a condition number for random
  linear conic systems.
\newblock {\em SIAM J. Optim.}, 15(4):1237--1261 (electronic), 2005.

\bibitem{CHL:09}
F.~Cucker, R.~Hauser, and M.~Lotz.
\newblock Adversarial smoothed analysis.
\newblock To appear in {\em J. of Complexity}.

\bibitem{cupe:02}
F.~Cucker and J.~Pe{\~n}a.
\newblock A primal-dual algorithm for solving polyhedral conic systems with a
  finite-precision machine.
\newblock {\em SIAM J. Optim.}, 12(2):522--554 (electronic), 2001/02.

\bibitem{CW01}
F.~Cucker and M.~Wschebor.
\newblock On the expected condition number of linear programming problems.
\newblock {\em Numer. Math.}, 94:419--478, 2002.

\bibitem{Demmel88}
J.~Demmel.
\newblock The probability that a numerical analysis problem is difficult.
\newblock {\em Math. Comp.}, 50:449--480, 1988.

\bibitem{DST}
J.~Dunagan, D.A. Spielman, and S.-H. Teng.
\newblock Smoothed analysis of condition numbers and complexity implications
  for linear programming.
\newblock {\em Math.~Programm.~Series~A}, 2009. To appear.

\bibitem{glas:95}
S.~Glasauer.
\newblock {\em Integralgeometrie konvexer K\"orper im sph\"arischen Raum}.
\newblock PhD thesis, Universit\"at Freiburg im Br., 1995.

\bibitem{goff:80}
J.-L. Goffin.
\newblock The relaxation method for solving systems of linear inequalities.
\newblock {\em Math. Oper. Res.}, 5(3):388--414, 1980.

\bibitem{HM:06}
R.~Hauser and T.~M{\"u}ller.
\newblock Conditioning of random conic systems under a general family of input
  distributions.
\newblock {\em Found. Comput. Math.}, 9(3):335--358, 2009.

\bibitem{howa:93}
R.~Howard.
\newblock The kinematic formula in {R}iemannian homogeneous spaces.
\newblock {\em Mem. Amer. Math. Soc.}, 106(509):vi+69, 1993.

\bibitem{rene:94}
J.~Renegar.
\newblock Some perturbation theory for linear programming.
\newblock {\em Math. Programming}, 65(1, Ser. A):73--91, 1994.

\bibitem{rene:95b}
J.~Renegar.
\newblock Incorporating condition measures into the complexity theory of linear
  programming.
\newblock {\em SIAM J. Optim.}, 5(3):506--524, 1995.

\bibitem{rene:95a}
J.~Renegar.
\newblock Linear programming, complexity theory and elementary functional
  analysis.
\newblock {\em Math. Programming}, 70(3, Ser. A):279--351, 1995.

\bibitem{sast:06}
A.~Sankar, D.A. Spielman, and S.H. Teng.
\newblock Smoothed analysis of the condition numbers and growth factors of
  matrices.
\newblock {\em SIAM J. Matrix Anal. Appl.}, 28(2):446--476 (electronic), 2006.

\bibitem{Schn:84}
R.~Schneider.
\newblock Smooth approximation of convex bodies.
\newblock {\em Rend. Circ. Mat. Palermo (2)}, 33(3):436--440, 1984.

\bibitem{Smale97}
S.~Smale.
\newblock Complexity theory and numerical analysis.
\newblock In A.~Iserles, editor, {\em Acta Numerica}, pages 523--551. Cambridge
  University Press, 1997.

\bibitem{ST:02}
D.A. Spielman and S.-H. Teng.
\newblock Smoothed analysis of algorithms.
\newblock In {\em Proceedings of the International Congress of Mathematicians},
  volume~I, pages 597--606, 2002.

\bibitem{ST:03}
D.A. Spielman and S.-H. Teng.
\newblock Smoothed analysis of termination of linear programming algorithms.
\newblock {\em Math. Programm. Series B}, 97:375--404, 2003.

\bibitem{ST:04}
D.A. Spielman and S.-H. Teng.
\newblock Smoothed analysis: Why the simplex algorithm usually takes polynomial
  time.
\newblock {\em Journal of the ACM}, 51:385--463, 2004.

\bibitem{spiv3:79}
M.~Spivak.
\newblock {\em A comprehensive introduction to differential geometry. {V}ol.
  {III}}.
\newblock Publish or Perish Inc., Wilmington, Del., second edition, 1979.

\bibitem{webs:94}
R.~Webster.
\newblock {\em Convexity}.
\newblock Oxford Science Publications. The Clarendon Press Oxford University
  Press, New York, 1994.

\bibitem{wend:62}
J.~G. Wendel.
\newblock A problem in geometric probability.
\newblock {\em Math. Scand.}, 11:109--111, 1962.

\bibitem{weyl:39}
H.~Weyl.
\newblock On the {V}olume of {T}ubes.
\newblock {\em Amer. J. Math.}, 61(2):461--472, 1939.

\end{thebibliography}

}


\end{document}